\documentclass[11pt]{amsart}
\usepackage{amsmath,amsthm,amssymb,amssymb, paralist, xspace, graphicx, url, amscd, euscript, mathrsfs,stmaryrd,epic,eepic,color, longtable}
\usepackage[all]{xy}
\usepackage{amsmath,amscd}
\usepackage{psfrag}
\usepackage{epsfig}
\usepackage{graphicx,transparent}
\usepackage{enumerate}
\usepackage{caption}
\usepackage{here}
\SelectTips{cm}{}

\usepackage[colorlinks=true]{hyperref}

\usepackage{pstricks}

\numberwithin{equation}{subsection}

1

\numberwithin{subsection}{section}

\allowdisplaybreaks[1]

\newtheorem*{namedtheorem}{\theoremname}
\newcommand{\theoremname}{testing}

\theoremstyle{plain}

\newtheorem{thm}{Theorem}[section]
\newtheorem{proposition}[thm]{Proposition}
\newtheorem{proposition-definition}[thm]{Proposition-Definition}
\newtheorem{lemma-definition}[thm]{Lemma-Definition}
\newtheorem{corollary}[thm]{Corollary}
\newtheorem{lemma}[thm]{Lemma}

\theoremstyle{definition}

\newtheorem{example}[thm]{Example}

\newtheorem{remark}[thm]{Remark}

\theoremstyle{remark}

\numberwithin{thm}{section}

\newcommand\ocM{\overline{\mathcal{M}}}

\newcommand\cC{\mathcal{C}}

\newcommand\cH{\mathcal{H}}

\newcommand\cL{\mathcal{L}}
\newcommand\cM{\mathcal{M}}

\newcommand\cO{\mathcal{O}}

\newcommand\cQ{\mathcal{Q}}

\newcommand\cS{\mathcal{S}}

\newcommand\cZ{\mathcal{Z}}

\newcommand\CC{\mathbb{C}}

\newcommand\bbP{\mathbb{P}}
\newcommand\PP{\mathbb{P}}

\newcommand\arr{\ifinner\to\else\longrightarrow\fi}

\def\displaytimes_#1{\mathrel{\mathop{\times}\limits_{#1}}}

\def\displayotimes_#1{\mathrel{\mathop{\bigotimes}\limits_{#1}}}

\newdir{ >}{{}*!/-5pt/@{>}}

\newcommand\doublelong[2]{\mathbin{\xymatrix{{}\ar@<3pt>[r]^{#1}
\ar@<-3pt>[r]_{#2}&}}}

\newlength{\ignora}

\newcommand{\ind}{\operatorname{Ind}}

\renewcommand{\setminus}{\smallsetminus}

\numberwithin{equation}{subsection}

\newcommand{\vmu}{\mu}

\newcommand{\BM}{\overline{\mathcal M}}

\newcommand{\tcM}{{\widetilde{\cM}}}

\newcommand{\ocS}{{\overline{\cS}}}

\newcommand{\ord}{\operatorname{ord}}

\newcommand{\Res}{\operatorname{Res}}

\newcommand{\hyp}{\operatorname{hyp}}

\newcommand{\odd}{\operatorname{odd}}

\newcommand{\even}{\operatorname{even}}

\newcommand{\nonhyp}{\operatorname{nonhyp}}

\begin{document}

\title[Principal boundary of differentials]{Principal boundary of moduli spaces of abelian and quadratic differentials}

\author{Dawei Chen}

\author{Qile Chen}

\thanks{D. Chen is partially supported by the NSF CAREER grant DMS-1350396 and Q. Chen is partially supported by the NSF grant DMS-1560830.}

\address[Chen]{Department of Mathematics\\
Boston College\\
Chestnut Hill, MA 02467\\
U.S.A.}
\email{dawei.chen@bc.edu}

\address[Chen]{Department of Mathematics\\
Boston College\\
Chestnut Hill, MA 02467\\
U.S.A.}
\email{qile.chen@bc.edu}

\date{\today}
\begin{abstract}
The seminal work of Eskin-Masur-Zorich described the principal boundary of moduli spaces of abelian differentials that parameterizes flat surfaces with a prescribed generic configuration of short parallel saddle connections. In this paper we describe the principal boundary for each configuration in terms of twisted differentials over Deligne-Mumford pointed stable curves. 
We also describe similarly the principal boundary of moduli spaces of quadratic differentials originally studied by Masur-Zorich. 
Our main technique is the flat geometric degeneration and smoothing developed by Bainbridge-Chen-Gendron-Grushevsky-M\"oller. 
\end{abstract}


\maketitle

\tableofcontents

%%%%%%%%%%%%%%%%%%%%%%%%%%%%%%%%%%%%%%%%%%%%%%%%%%%%%%%%%%%%%%%%%%%%%%%%%%%%%%%%%%%%%%%%%%%%%%%%%%%

%%%%%%%%%%%%%%%%%%%%%%%%%%%%
%%%%%%%%%%%%%%%%%%%%%%%%%%%%
\section{Introduction}
\label{sec:intro}
%%%%%%%%%%%%%%%%%%%%%%%%%%%%
%%%%%%%%%%%%%%%%%%%%%%%%%%%%

Many questions about Riemann surfaces are related to study their flat structures induced from abelian differentials, where the zeros of differentials correspond to the saddle points of flat surfaces. Loci of abelian differentials with prescribed type of zeros form a natural stratification of the moduli space of abelian differentials. These strata have fascinating geometry and can be applied to study dynamics on flat surfaces. 

Given a configuration of saddle connections for a stratum of flat surfaces, Veech and Eskin-Masur (\cite{Veech, EM}) showed that the number of collections of saddle connections with bounded lengths has quadratic asymptotic growth, whose leading coefficient is called the Siegel-Veech constant for this configuration. Eskin-Masur-Zorich (\cite{EMZ}) gave a complete description of all possible configurations of parallel saddle connections on a generic flat surface. They further provided a recursive method to calculate the corresponding Siegel-Veech constants. To perform this calculation, a key step is to describe the \emph{principal boundary} whose tubular neighborhood parameterizes flat surfaces with short parallel saddle connections for a given configuration. 

As remarked in \cite{EMZ}, flat surfaces contained in the Eskin-Masur-Zorich principal boundary can be disconnected and have total genus smaller than that of the original stratum. Therefore, as the underlying complex curves degenerate by shrinking the short saddle connections, the Eskin-Masur-Zorich principal boundary does not directly imply the limit objects from the viewpoint of algebraic geometry. In this paper we solve this problem by describing the principal boundary in the setting of the strata compactification \cite{BCGGM1} and consequently in the Deligne-Mumford compactification. 

\subsection*{Main Result} {\em For each configuration we give a complete description for the principal boundary in terms of twisted differentials over pointed stable curves.}  
\smallskip

This result is a combination of Theorems~\ref{thm:principal-I} and~\ref{thm:principal-II}. Along the way we deduce some interesting consequences about meromorphic differentials on $\bbP^1$ that admit the same configuration (see Propositions~\ref{prop:unique-moduli-typeI} and~\ref{prop:unique-moduli-typeII}). Moreover, when a stratum contains connected components due to spin or hyperelliptic structures (\cite{KZ}), Eskin-Masur-Zorich (\cite{EMZ}) described how to distinguish these structures nearby the principal boundary via an analytic approach. Here we provide algebraic proofs for the distinction of spin and hyperelliptic structures in the principal boundary under our setting (see Sections~\ref{subsec:parity-I} and~\ref{subsec:parity-II} for related results).

Masur-Zorich (\cite{MZ}) described similarly the principal boundary of strata of quadratic differentials. Our method can also give a description of the principal boundary in terms of twisted quadratic differentials in the sense of \cite{BCGGM3} (see Section~\ref{sec:quad} for details). 

Twisted differentials play an important role in our description of the principal boundary, so we briefly recall their definition (see \cite{BCGGM1} for more details). Given a zero type $\vmu = (m_1, \ldots, m_n)$, a \emph{twisted differential} $\eta$ of type $\vmu$ on an $n$-pointed stable curve $(C, \sigma_1, \ldots, \sigma_n)$ is a collection of (possibly meromorphic) differentials $\eta_i$ on each irreducible component $C_i$ of $C$, satisfying the following conditions: 
\begin{itemize}
\item[(0)] $\eta$ has no zeros or poles away from the nodes and markings of $C$ and $\eta$ has the prescribed zero order $m_i$ at each marking $\sigma_i$. 
\item[(1)] If a node $q$ joins two components $C_1$ and $C_2$, then $\ord_{q} \eta_1 + \ord_q \eta_2 = -2$. 
\item[(2)] If $\ord_{q} \eta_1 = \ord_q \eta_2 = -1$, then $\Res_q \eta_1 + \Res_q \eta_2 = 0$. 
\item[(3)] If $C_1$ and $C_2$ intersect at $k$ nodes $q_1, \ldots, q_k$, then $\ord_{q_i} \eta_1 - \ord_{q_i} \eta_2$ are either all positive, or all negative, or all equal to zero for $i = 1, \ldots, k$. 
\end{itemize}

Condition (3) provides a partial order between irreducible components that are not disjoint. If one expands it to a full order between all irreducible components of $C$, then there is an extra global residue condition which governs when such twisted differentials are limits of abelian differentials of type $\vmu$. A construction of the moduli space of twisted differentials can be found in \cite{BCGGM2}. 

By using $\eta$ on all maximum components and forgetting its scales on components of smaller order, \cite{BCGGM1} describes a strata compactification in the Hodge bundle over the Deligne-Mumford moduli space $\BM_{g,n}$. As remarked in \cite{BCGGM1}, if one forgets $\eta$ and only keeps track of the underlying pointed stable curve $(C, \sigma_1, \ldots, \sigma_n)$, it thus gives the (projectivized) strata compactification in $\BM_{g,n}$. Hence our description of the principal boundary in terms of twisted differentials determines the corresponding boundary in the Deligne-Mumford compactification. To illustrate our results, we will often draw such stable curves in the Deligne-Mumford boundary. 

For an introduction to flat surfaces and related topics, we refer to the surveys \cite{Zorich, Wright, Bootcamp}. Besides \cite{BCGGM1}, there are several other strata compactifications, see \cite{FP} for an algebraic viewpoint, \cite{Guere, CC} for a log geometric viewpoint and \cite{MW} for a flat geometric viewpoint. Algebraic distinctions of spin and hyperelliptic structures in the boundary of strata compactifications are also discussed in \cite{Gendron, ChenDiff, CC}. 

This paper is organized as follows. In Sections~\ref{sec:typeI} and~\ref{sec:typeII} we describe the principal boundary of type I and of type II, respectively, following the roadmap of \cite{EMZ}. In Section~\ref{sec:parity} we provide algebraic arguments for distinguishing spin and hyperelliptic structures in the principal boundary. Finally in Section~\ref{sec:quad} we explain how one can describe the principal boundary of strata of quadratic differentials by using twisted quadratic differentials. Throughout the paper we also provide a number of examples and figures to help the reader quickly grasp the main ideas. 

\subsection*{Notation} We denote by $\vmu$ the singularity type of differentials, by $\cH(\vmu)$ the stratum of abelian differentials of type $\vmu$ and by $\cQ(\vmu)$ the stratum of quadratic differentials of type $\vmu$. An $n$-pointed stable curve is generally denoted by $(C, \sigma_1, \ldots, \sigma_n)$. We use $(C, \eta)$ to denote a twisted differential on $C$. The underlying divisor of a differential $\eta$ is denoted by $(\eta)$. Configurations of saddle connections are denoted by $\cC$ and all configurations considered in this paper are admissible in the sense of \cite{EMZ}. 

\subsection*{Acknowledgements} We thank Matt Bainbridge, Alex Eskin, Quentin Gendron, Sam Grushevsky, Martin M\"oller, and Anton Zorich for inspiring discussions on related topics.

%%%%%%%%%%%%%%%%%%%%%%%%%%%%
%%%%%%%%%%%%%%%%%%%%%%%%%%%%
\section{Principal boundary of type I}
\label{sec:typeI}
%%%%%%%%%%%%%%%%%%%%%%%%%%%%
%%%%%%%%%%%%%%%%%%%%%%%%%%%%

%%%%%%%%%%%%%%%%%%%%%%%%
\subsection{Configurations of type I: saddle connections joining distinct zeros}
%%%%%%%%%%%%%%%%%%%%%%%%

Let $C$ be a flat surface in $\cH(\vmu)$ with two chosen zeros $\sigma_1$ and $\sigma_2$ of order $m_1$ and $m_2$, respectively. 
Suppose $C$ has precisely $p$ homologous saddle connections 
$\gamma_1, \ldots, \gamma_p$ joining $\sigma_1$ and $\sigma_2$ such that the following conditions hold:
\begin{itemize}
\item All saddle connections $\gamma_{i}$ are oriented from $\sigma_{1}$ to $\sigma_{2}$ with identical holonomy vectors.
\item The cyclic order of $\gamma_1, \ldots, \gamma_p$ at $\sigma_1$ is clockwise. 
\item The angle between $\gamma_i$ and $\gamma_{i+1}$ is $2\pi (a'_i+1)$ at $\sigma_1$ and $2\pi (a''_i+1)$ at $\sigma_2$, where $a'_i, a''_i \geq 0$. 
\end{itemize}
Then we say that $C$ has a {\em configuration of type} $\cC = (m_1, m_2, \{a'_i, a''_i\}_{i=1}^p)$. We emphasis here that this configuration $\cC$ is defined with the two chosen zeros $\sigma_{1}$ and $\sigma_{2}$. If $p = 1$, we also denote the configuration by $\cC = (m_1, m_2)$ for simplicity. 
Since the cone angle at $\sigma_i$ is $2\pi (m_i + 1)$ for $i=1,2$, we necessarily have 
\begin{eqnarray}\label{eqn:a-and-m}
\sum_{i=1}^p (a'_i+1) = m_1 + 1 \quad \text{and}\quad \sum_{i=1}^p (a''_i+1) = m_2 + 1. 
\end{eqnarray}

%%%%%%%%%%%%%%%%%%%%%%%%
\subsection{Graphs of configurations}
%%%%%%%%%%%%%%%%%%%%%%%%

Given two fixed zeros $\sigma_{1}$ and $\sigma_{2}$ and a configuration $\cC = (m_1, m_2, \{a'_i, a''_i\}_{i=1}^p)$ as in the previous section, to describe the dual graphs of the underlying nodal curves in the principal boundary of twisted differentials, we introduce the {\em configuration graph} $G(\cC)$ as follows: 
\begin{enumerate}
 \item The set of vertices is $\{v_{R}, v_{1}, \cdots, v_{p}\}$.  
 
 \item The set of edges is $\{l_{1}, \cdots, l_{p}\}$, where each $l_{i}$ joins $v_{i}$ and $v_{R}$.

 \item We associate to $v_R$ the subset of markings $L_R = \{\sigma_{1}, \sigma_{2}\}$ and to each $v_{i}$ a subset of markings $L_i \subset \{\sigma_{j}\}$ such that $L_{R} \sqcup L_1 \sqcup \cdots \sqcup L_p$
is a partition of $\{\sigma_1, \ldots, \sigma_n\}$. 

 \item We associate to each $v_{i}$ a positive integer $g(v_{i})$ such that
 $$\sum_{i=1}^{p} g(v_{i}) = 2g-2 \quad \text{and} \quad \sum_{\sigma_{j} \in L_i}\mu_j + (a'_i + a''_i + 1) = 2g(v_i) - 2.$$
\end{enumerate}

Figure~\ref{fig:graph-I} shows a pointed nodal curve whose dual graph is of type $G(\cC)$: 
\begin{figure}[h]
    \centering
    \psfrag{R}{$R$}
    \psfrag{A}{$C_1$}
    \psfrag{B}{$C_p$}
    \psfrag{a}{$\sigma_1$}
    \psfrag{b}{$\sigma_2$}
    \psfrag{L}{$L_1$}
    \psfrag{M}{$L_p$}   
    \includegraphics[scale=0.8]{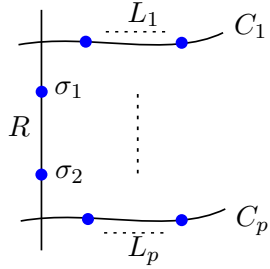}
    \caption{\label{fig:graph-I} A curve with dual graph of type $\cC$.}
    \end{figure}

%%%%%%%%%%%%%%%%%%%%%%%%%%%%%%%%
\subsection{The principal boundary of type I}\label{ss:principal-boundaryI}
%%%%%%%%%%%%%%%%%%%%%%%%%%%%%%%%%

Denote by $\Delta(\vmu, \cC)$ the space of twisted differentials $\eta$ satisfying the following conditions: 

\begin{itemize}
\item The underlying dual graph of $\eta$ is given by $G(\cC)$, with nodes $q_{i}$ and components $C_{i}$ corresponding to $l_{i}$ and $v_{i}$, respectively.

\item The component $R$ corresponding to the vertex $v_{R}$ is isomorphic to $\bbP^1$ and contains only $\sigma_1$ and $\sigma_2$ among all the markings. 

\item Each $C_{i}$ has markings labeled by $L_{i}$ and has genus equal to $g(v_i)$. 

\item For each $i=1, \ldots, p$, $\ord_{q_i} \eta_{C_i} = a'_i + a''_i$ and $\ord_{q_i} \eta_{R}  = -  a'_i - a''_i - 2$. 

\item For each $i=1, \ldots, p$, $\Res_{q_i} \eta_R = 0$. 

\item $\eta_R$ admits the configuration $\cC$ of saddle connections from $\sigma_{1}$ to $\sigma_{2}$.
\end{itemize}

Recall that the twisted differential $\eta$ defines a flat structure on $R$ (up to scale). Thus it makes sense to talk about the configuration $\cC$ on $R$.
We say that $\Delta(\vmu,\cC)$ is the {\em principal boundary} associated to the configuration $\cC$. 

\smallskip

Suppose $C^{\varepsilon}\in \cH(\vmu)$ has the configuration $\cC = (m_1, m_2, \{a'_i, a''_i\}_{i=1}^p)$ such that the $p$ homologous saddle connections $\gamma_1, \ldots, \gamma_p$ of $C$ have length at most $\varepsilon $. We want to 
determine the limit twisted differential as the length of all $\gamma_i$ shrinks to zero. To avoid further degeneration, suppose that $C^{\varepsilon}$ does not have any other saddle connections shorter than $3\varepsilon$ (the locus of such $C^{\varepsilon}$ is called the {\em thick part} of the configuration $\cC$ in \cite{EMZ}). Take a small disk under the flat metric such that it contains 
$\sigma_1$, $\sigma_2$, all $\gamma_i$, and no other zeros (see \cite[Figure 5]{EMZ}). Within this disk, shrink $\gamma_i$ to zero while keeping the configuration $\cC$, such that all other periods become arbitrarily large compared to $\gamma_i$.  

\begin{thm}
\label{thm:principal-I}
The limit twisted differential of $C^{\varepsilon}$ as $\gamma_i \to 0$ is contained in $\Delta(\vmu, \cC)$. Conversely, twisted differentials in $\Delta(\vmu, \cC)$ can be smoothed to of type $C^{\varepsilon}$. 
\end{thm}

\begin{proof}
Since $\gamma_i$ and $\gamma_{i+1}$ are homologous and next to each other, they bound a surface $C_i^{\varepsilon}$ with $\gamma_i$ and $\gamma_{i+1}$
as boundary, see \cite[Figure 5]{EMZ}. The inner angle between $\gamma_i$ and $\gamma_{i+1}$ at $\sigma_1$ is $2\pi (a'_i+1)$ and at $\sigma_2$ is $2\pi (a''_i+1)$. Shrinking the $\gamma_j$ to zero under the flat metric, the limit of $C_i^{\varepsilon}$ forms a flat surface $C_i$, and denote by $q_i$ the limit position of $\sigma_1$ and $\sigma_2$ in $C_i$. This shrinking operation is the inverse of breaking up a zero, see \cite[Figure 3]{EMZ}, which implies that 
the cone angle at $q_i$ is $2\pi (a'_i+a''_i+1)$, hence $C_i$ has a zero of order $a_i' + a_i''$ at $q_i$. 

On the other hand, instead of shrinking the $\gamma_j$, up to scale it amounts to expanding the other periods of $C_i^{\varepsilon}$ arbitrarily long compared to the $\gamma_j$. Since a small neighborhood $N_i$ enclosing both $\gamma_i$ and $\gamma_{i+1}$ in $C_i^{\varepsilon}$ consists of $2 (a'_i + a''_i +1)$ metric half-disks, under the expanding operation they turn into 
$2 (a'_i + a''_i +1)$ metric half-planes that form the basic domain decomposition for a pole of order $a'_i + a''_i + 2$ in the sense of \cite{Boissy}. Moreover, the boundary loop of $N_i$ corresponds to the vanishing cycle around $q_i$ in the shrinking operation, which implies that the resulting pole will be glued to $q_i$ as a node in the limit stable curve, hence we still use $q_i$ to denote the pole. See Figure~\ref{fig:pole-I} for the case $p=2$ and $m_1 = m_2 = 0$. 
\begin{figure}[h]
    \centering
    \psfrag{r}{$\gamma_1$}
    \psfrag{s}{$\gamma_2$}
    \psfrag{1}{$L_1^+$}
    \psfrag{2}{$R_1^+$}
    \psfrag{3}{$L_1^-$}
    \psfrag{4}{$R_1^-$}
    \psfrag{5}{$L_2^+$}
    \psfrag{6}{$R_2^+$}
    \psfrag{7}{$L_2^-$}
    \psfrag{8}{$R_2^-$}
    \includegraphics[scale=1.0]{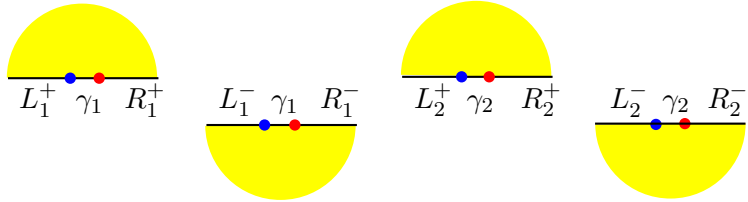}
    \caption{\label{fig:pole-I} The flat geometric neighborhood of $\gamma_1$ and $\gamma_2$ for the case $p=2$ and $m_1 = m_2 = 0$. Here we identify $L_1^- = L_2^+$, $L_1^+ = L_2^-$, $R_1^+ = R_2^-$, and $R_1^- = R_2^+$. As $\gamma_1, \gamma_2 \to 0$, the middle two half-disks form a neighborhood of an ordinary point and the remaining two half-disks form a neighborhood of another ordinary point. Alternatively as $L_i^{\pm}$ and $R_j^{\pm} \to \infty$, the middle two half-planes form a neighborhood of a double pole and the remaining two half-planes form a neighborhood of another double pole. Both poles have zero residue.}
    \end{figure}

Let $(R, \eta_R)$ be the limit meromorphic differential out of the expanding operation. We thus conclude that 
$$ (\eta_R) = m_1 \sigma_1 + m_2 \sigma_2 - \sum_{i=1}^p (a'_i + a''_i + 2) q_i. $$
By the relation~\eqref{eqn:a-and-m}, the genus of $R$ is zero, hence $R\cong \PP^1$. Since $q_i = C_i\cap R$ is a separating node, it follows from the global residue condition of \cite{BCGGM1} that $\Res_{q_i} \eta_R = 0$. Finally, in the expanding process the saddle connections $\gamma_i$ are all fixed, hence the configuration $\cC$ is preserved in the limit meromorphic differential $\eta_R$. Summarizing the above discussion, we see that 
the limit twisted differential is parameterized by $\Delta(\vmu, \cC)$. 

The other part of the claim follows from the flat geometric smoothing of \cite{BCGGM1}, as twisted differentials in $\Delta(\vmu, \cC)$ satisfy the global residue condition and have the desired configuration of saddle connections.  
\end{proof}

\begin{remark} For the purpose of calculating Siegel-Veech constants, the Eskin-Masur-Zorich principal boundary only takes into account the non-degenerate components $C_i$ and discards the degenerate rational component $R$, though it is quite visible --- for instance, $R$ can be seen as the central sphere in \cite[Figure 5]{EMZ}. 
\end{remark}

%%%%%%%%%%%%%%%%%%%%%%%%%%%%
\subsection{Meromorphic differentials of type I on $\PP^1$}
%%%%%%%%%%%%%%%%%%%%%%%%%%%%

Recall that for a twisted differential $\eta$ in $\Delta(\vmu, \cC)$, its restriction $\eta_R$ on the component $R\cong \PP^1$ has two zeros and $p$ poles, where the residue at each pole is zero. Up to scale, $\eta_R$ is uniquely determined by the zeros and poles. In this section we study the locus of $\PP^1$ marked at such zeros and poles. 

Given integers $m_1, m_2 \geq 1$ and $n_1, \ldots, n_p \geq 2$ with $m_1 + m_2 - \sum_{i=1}^p n_i = -2$, let $\cZ \subset \mathcal{M}_{0, p+2}$ be the locus of pointed rational curves $(\PP^1, \sigma_1, \sigma_2, q_1, \ldots, q_p)$ such that there exists a differential $\eta_0$ on $\PP^1$ satisfying that 
$$(\eta_0) = m_1 \sigma_1 + m_2 \sigma_2 - \sum_{i=1}^p n_i q_i \quad \text{and} \quad \Res_{q_i} \eta_0 = 0$$ 
for each $i= 1, \ldots, p$. 

For a given (admissible) configuration $\cC=  (m_1, m_2, \{a'_i, a''_i\}_{i=1}^p)$, consider the subset $\cZ(\cC) \subset \cZ$ parameterizing differentials $\eta_0$ on $\PP^1$ (up to scale) that admit a configuration of type $\cC$. 

\begin{proposition}
\label{prop:unique-moduli-typeI}
$\cZ(\cC)$ consists of a single point. 
\end{proposition}  

\begin{proof}
We provide a constructive proof using the flat geometry of meromorphic differentials. Let us make some observation first. Suppose $\eta_0$ is a differential on $\PP^1$ whose underlying divisor
corresponds to a point in $\cZ$. Since $\eta_0$ has zero residue at every pole, for any closed path $\gamma$ that does not contain a pole of $\eta_0$, the Residue Theorem says that 
$$ \int_{\gamma} \eta_0 = 0. $$
In particular, if $\alpha$ and $\beta$ are two saddle connections joining $\sigma_1$ to $\sigma_2$, then $\alpha - \beta$ represents a closed path on $\PP^1$, hence 
$$ \int_{\alpha} \eta_0 = \int_{\beta} \eta_0, $$
and $\alpha$ and $\beta$ necessarily have the same holonomy. It also implies that $\eta_0$ has no self saddle connections.

Now suppose $\eta_0$ admits a configuration of type $\cC$, i.e., up to scale it corresponds to a point in $\cZ(\cC)$. 
Recall that $\sigma_1$, $\sigma_2$, and $q_i$ are the zeros and poles of order $m_1$, $m_2$, and $a'_i + a''_i +2$, respectively, where $i= 1,\ldots, p$, and $\gamma_1, \ldots, \gamma_p$ are the saddle connections joining $\sigma_1$ to $\sigma_2$ such that the angle between $\gamma_i$ and $\gamma_{i+1}$ in the clockwise orientation at $\sigma_1$ is $2\pi (a'_i +1)$, and at $\sigma_2$ is $2\pi (a''_i + 1)$. By the preceding paragraph, there are no other saddle connections between $\sigma_1$ and $\sigma_2$. 

Rescale $\eta_0$ such that all the $\gamma_i$ have holonomy equal to $1$, that is, they are in horizontal, positive direction, and of length $1$. Cut the flat surface $\eta_0$ along all horizontal directions through $\sigma_1$ and $\sigma_2$, such that $\eta_0$ is decomposed into a union of half-planes as basic domains in the sense of \cite{Boissy}. These basic domains are of two types according to their boundary half-lines and saddle connections. The boundary of the basic domains of the first type contains exactly one of 
$\sigma_1$ and $\sigma_2$ that emanates two half-lines to infinity on both sides. The boundary of the basic domains of the second type, from left to right, consists of a half-line ending at $\sigma_1$, followed by a saddle connection $\gamma_i$, and then a half-line emanating for $\sigma_2$. 

Since the angle between $\gamma_i$ and $\gamma_{i+1}$ is given for each $i$, the configuration $\cC$ determines how these basic domains are glued together to form $\eta_0$. More precisely, start from an upper half-plane $S_1^{+}$ of the second type with two boundary half-lines $L_1^{+}$ to the left and $R_1^+$ to the right, joined by the saddle connection $\gamma_1$. Turn around 
$\sigma_1$ in the clockwise orientation. Then we will see a lower half-plane $S_1^{-}$ of the second type with two boundary half-lines $L_1^{-}$ and $R_1^{-}$ joined by $\gamma_1$. If $a'_1 = 0$, i.e., if the angle between $\gamma_1$ and $\gamma_2$ in the clockwise orientation is $2\pi$, then next we will see an upper half-plane $S_2^{+}$ of the second type with two boundary half-lines $L_2^{+}$ and $R_2^{+}$ joined by $\gamma_2$, which is glued to $S_1^{-}$ by identifying $L_2^{+}$ with $L_1^{-}$. See Figure~\ref{fig:pole-I} above for an illustration of this case. 

On the other hand if $a'_1 > 0$, we will see $a'_1$ pairs of upper and lower half-planes of the first type containing only $\sigma_1$ in their boundary, and then followed by the upper half-plane of the second type containing $\gamma_2$ in the boundary. Repeat this process for each pair $\gamma_i$ and $\gamma_{i+1}$ consecutively, and also use the angle between $\gamma_i$ and $\gamma_{i+1}$ at $\sigma_2$ to determine the identification of the $R_i^{\pm}$-edges emanated from $\sigma_2$. We conclude that the gluing pattern of these half-planes is uniquely determined by the configuration $\cC$. 

Finally, since the angle between $\gamma_i$ and $\gamma_{i+1}$ at $\sigma_1$ is $2\pi (a'_i + 1)$ and at $\sigma_2$ is $2\pi (a''_i + 1)$, it determines precisely $a'_i + a''_i + 1$ pairs of upper and lower half-planes that share the same point at infinity. In other words, they form a flat geometric neighborhood of a pole with order $a'_i + a''_i + 2$, which is the desired pole order of $q_i$ for $i=1,\ldots, p$. 
\end{proof}

\begin{corollary}
The cardinality of $\cZ$ is equal to the number of integral tuples 
$\{a'_i, a''_i \}_{i=1}^p$ where $a'_i, a''_i \geq 0$ and $a'_i + a''_i + 2 = n_i$ for each $i$. 
\end{corollary}

\begin{proof}
Such tuples have a one-to-one correspondence with all (admissible) configurations with the given zero and pole orders $m_1, m_2, n_1, \ldots, n_p$,
hence the claim follows from Proposition~\ref{prop:unique-moduli-typeI}. 
\end{proof}

\begin{example}
Consider the case $m_1 = 1$, $m_2 = 1$, $n_1 = 2$ and $n_2 = 2$. The only admissible configuration is 
$$ a'_1 = a''_1 = a'_2 = a''_2 = 0, $$
hence $\cZ$ consists of a single point. As a cross check, take $\sigma_1 = 1$, $q_1 = 0$, and $q_2 = \infty$ in $\PP^1$, and let $z$ be the affine coordinate. Then up to scale $\eta_0$ can be written as  
$$ \frac{(z-1)(z-\sigma_2)}{z^2} dz. $$
It is easy to see that $\Res_{\sigma_i} \eta_0 = 0$ if and only if $\sigma_2 = -1$.  
\end{example}

\begin{example}
Consider the case $m_1 = 1$, $m_2 = 3$ and $n_1 = n_2 = n_3 = 2$. There do not exist nonnegative integers 
$a'_1, a'_2, a'_3$ satisfying that 
$$ (a'_1 +1) + (a'_2 + 1) + (a'_3+1) = m_1 + 1 = 2, $$
because the left-hand side is at least $3$. Since there is no admissible configuration, we conclude that $\cZ$ is empty. As a cross check, 
let $q_1 = 0$, $q_2 = 1$, and $q_3 = \infty$. Up to scale $\eta_0$ can be written as
$$ \frac{(z-\sigma_1)(z-\sigma_2)^3}{z^2 (z-1)^2} dz. $$
One can directly verify that there are no $\sigma_1, \sigma_2 \in \PP^1\setminus \{0, 1, \infty \}$ such that 
$\Res_{\sigma_i} \eta_0 = 0$. 
\end{example}

%%%%%%%%%%%%%%%%%%%%%%%%%%%%
%%%%%%%%%%%%%%%%%%%%%%%%%%%%
\section{Principal boundary of type II}
\label{sec:typeII}
%%%%%%%%%%%%%%%%%%%%%%%%%%%%
%%%%%%%%%%%%%%%%%%%%%%%%%%%%

%%%%%%%%%%%%%%%%%%%%%%%%
\subsection{Configurations of type II: saddle connections joining a zero to itself}
%%%%%%%%%%%%%%%%%%%%%%%%

Let $C$ be a flat surface in $\cH(\vmu)$. Suppose $C$ has precisely $m$ homologous closed saddle connections 
$\gamma_1, \ldots, \gamma_m$, each joining a zero to itself. Let $L\subset \{ 1, \ldots, m \}$ be an index subset such that 
the curves $\gamma_l$ for $l\in L$ bound $q$ cylinders. After removing the cylinders along with all the $\gamma_k$, the remaining part in $C$ 
splits into $p = m - q$ disjoint surfaces $C_1, \ldots, C_p$, where the boundary of the closure $\overline{C}_k$ of each $C_k$ consists of two 
closed saddle connections $\alpha_k$ and $\beta_k$. These surfaces are glued together in a cyclic order to form $C$. More precisely, each $C_k$ 
 is connected to $C_{k+1}$ by either identifying $\alpha_k$ with $\beta_{k+1}$ (as some $\gamma_i$ in $C$) or inserting a metric cylinder with boundary $\alpha_k$ and $\beta_{k+1}$. The sum of genera of the $C_k$ is $g-1$, because the cyclic gluing procedure creates a central handle, hence it adds an extra one to the total genus 
 (see \cite[Figure 7]{EMZ}). 

There are two types of the surfaces $C_k$ according to their boundary components. If the boundary saddle connections $\alpha_i$ and $\beta_i$ of 
$\overline{C}_i$ are disjoint, we say that $C_i$ has \emph{a pair of holes} boundary. In this case $\alpha_i$ contains a single zero $z_i$ with cone angle $(2a_i + 3) \pi$ inside $C_i$, and $\beta_i$ contains a single zero $w_i$ with cone angle $(2b_i + 3) \pi$ inside $\overline{C}_i$, 
where $a_i, b_i \geq 0$. We also take into account the special case $m=1$, i.e., when we cut $C$ along $\gamma_1$, we get only one surface $C_1$ with two disjoint boundary components $\alpha_1$ and $\beta_1$. In this case $z_1$ is identified with $w_1$ in $C$, and we still say that $C_1$ has a pair of holes boundary. 

For the remaining case, if $\alpha_j$ and $\beta_j$ form a connected component for the boundary of $\overline{C}_j$, we say that $C_j$ has a \emph{figure eight} boundary. In this case 
$\alpha_j$ and $\beta_j$ contain the same zero $z_j$. Denote by $2(c'_j+1)\pi$ and $2(c''_j+1)\pi$ the two angles bounded by $\alpha_j$ and $\beta_j$ inside $C_j$, where $c'_j, c''_j \geq 0$, and let $c_j = c'_j + c''_j$.

In summary, the configuration considered above consists of the data 
$$ (L, \{ a_i, b_i\}, \{ c'_j, c''_j \}). $$

Conversely, given the surfaces $C_k$ along with some metric cylinders, local gluing patterns can create zeros of the following three types (see \cite[Figure 12]{EMZ} and 
\cite[Figures 6-8]{BG}): 
\begin{itemize}
\item[(i)] A cylinder, followed by $k\geq 1$ surfaces $C_1, \ldots, C_k$, each of genus $g_i \geq 1$ with a figure eight boundary, followed 
by a cylinder. The total angle at the newborn zero is 
$$ \pi + \sum_{i=1}^k (2c'_i + 2c''_i + 4)\pi + \pi, $$
hence its zero order is 
$$\sum_{i=1}^k (c_i+ 2).$$ 
 
\item[(ii)] A cylinder, followed by $k\geq 0$ surfaces $C_i$, each of genus $g_i \geq 1$ with a figure eight boundary, followed by a surface 
$C_{k+1}$ of genus $g_{k+1} \geq 1$ with a pair of holes boundary. The total angle at the newborn zero is 
$$ \pi + \sum_{i=1}^k (2c'_i + 2c''_i + 4)\pi +  (2b_{k+1} + 3) \pi, $$
hence its zero order is 
$$\sum_{i=1}^k (c_i  + 2) + (b_{k+1}+1).$$ 

\item[(iii)] A surface $C_0$ of genus $g_0 \geq 1$ with a pair of holes boundary, followed by $k\geq 0$ surfaces $C_i$, each of genus $g_i \geq 1$ with a figure eight boundary, followed by a surface $C_{k+1}$ of genus $g_{k+1} \geq 1$ with a pair of holes boundary. The total angle at the newborn zero is 
$$ (2a_0 + 3)\pi + \sum_{i=1}^k (2c'_i + 2c''_i + 4)\pi +  (2b_{k+1} + 3) \pi, $$
hence its zero order is 
$$\sum_{i=1}^k (c_i  + 2) + (a_0+1) + (b_{k+1}+1).$$ 
\end{itemize} 

For example, the flat surface in \cite[Figure 7]{EMZ} is constructed as follows: $S_1$ with a pair of holes boundary, followed by $S_2$ with a pair 
of holes boundary, then a cylinder, followed by $S_3$ with a figure eight boundary, then another cylinder, followed by $S_4$ with a figure eight boundary, and finally back to $S_1$. 

%%%%%%%%%%%%%%%%%%%%%%%%
\subsection{The principal boundary of type II}
%%%%%%%%%%%%%%%%%%%%%%%%

Suppose $C^{\varepsilon}\in \cH(\vmu)$ has the configuration $\cC = (L, \{ a_i, b_i\}, \{ c'_j, c''_j \})$ with the $m$ homologous saddle connections 
$\gamma_1, \ldots, \gamma_m$ of length at most $\varepsilon$. Moreover, suppose that $C^{\varepsilon}$ does not have any other saddle connections shorter than $3\varepsilon$. As before, we degenerate $C^{\varepsilon}$ by shrinking $\gamma_i$ to zero while keeping the configuration, such that the ratio of any other period to $\gamma_i$ becomes arbitrarily large. Let $\Delta(\vmu, \cC)$ be the space of twisted differentials that arise as limits of such a degeneration process. Recall the three types 
of gluing patterns and newborn zeros in the preceding section. We will analyze the types of their degeneration as building blocks to describe twisted differentials in $\Delta(\vmu, \cC)$. 

For the convenience of describing the degeneration, we view a cylinder as a union of two half-cylinders by truncating it in the middle. Then as its height tends to be arbitrarily large compared to the width, each half-cylinder becomes a half-infinite cylinder, which represents a flat geometric neighborhood of a simple pole. Moreover, the two newborn simple poles have opposite residues, because the two half-infinite cylinders have the same width with opposite orientations.   

\begin{proposition} 
\label{prop:II-(i)}
Consider a block of surfaces of type (i) in $C^{\varepsilon}$, that is, a half-cylinder, followed by $k\geq 1$ surfaces $C_1^{\varepsilon}, \ldots, C_k^{\varepsilon}$, each of genus $g_i \geq 1$ with a figure eight boundary, followed by a half-cylinder. Let $\sigma$ be the newborn zero of 
order $\sum_{i=1}^k (c_i+ 2)$. As $\varepsilon \to 0$, we have 
\begin{itemize}
\item The limit differential consists of $k$ disjoint surfaces $C_1, \ldots, C_k$ attached to a component $R\cong \PP^1$ 
at the nodes $q_1, \ldots, q_k$, respectively. 
\item $R$ contains only $\sigma$ among all the markings. 
\item For each $i = 1, \ldots, k$, $\ord_{q_i} \eta_{C_i} = c_i$ and $\ord_{q_i}\eta_R = - c_i - 2$. 
\item For each $i = 1, \ldots, k$, $\Res_{q_i}\eta_R = 0$. 
\item $\eta_R$ has two simple poles at $q_0$ and $q_{k+1} \in R\setminus \{\sigma, q_1, \ldots, q_k \}$ with opposite residues $\pm r$. 
\item $\eta_R$ admits a configuration of type (i), i.e., it has precisely $k+1$ homologous self saddle connections with angles $2(c'_i+1)\pi$ and $2(c''_i+1)\pi$ in between consecutively for $i=1, \ldots, k$, and with holonomy equal to $r$ up to sign. 
\end{itemize}
\end{proposition}

See Figure~\ref{fig:II-(i)} for an illustration of the underlying curve of the limit differential. 

\begin{figure}[h]
    \centering
    \psfrag{R}{$R$}
    \psfrag{A}{$C_1$}
    \psfrag{B}{$C_k$}
    \psfrag{s}{$\sigma$}
    \psfrag{0}{$q_0$}
    \psfrag{1}{$q_1$}
    \psfrag{k}{$q_k$}
    \psfrag{l}{$q_{k+1}$}  
    \includegraphics[scale=0.8]{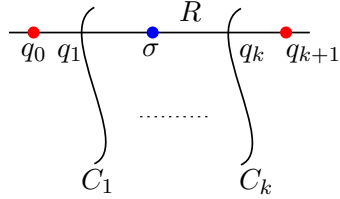}
    \caption{\label{fig:II-(i)} The underlying curve of the limit differential in Proposition~\ref{prop:II-(i)}.}
    \end{figure}

\begin{proof}
As $\varepsilon \to 0$, the limit of each $C_i^{\varepsilon}$ is a flat surface $C_i$, where the figure eight boundary of $C_i^{\varepsilon}$ shrinks 
to a single zero $q_i$ with cone angle $( 2 c_i + 2)\pi$, i.e., $q_i$ is a zero of order $c_i$. This shrinking operation is the inverse of 
the figure eight construction, see \cite[Figure 10]{EMZ}. On the other hand, instead of shrinking the boundary saddle connections 
$\alpha_i, \beta_i$ of the $C_i^{\varepsilon}$, up to scale it amounts to expanding the other periods of the $C_i^{\varepsilon}$ arbitrarily 
long compared to the $\alpha_i, \beta_i$. Since a small neighborhood $N_i$ enclosing both $\alpha_i$ and $\beta_i$ in $C_i^{\varepsilon}$ 
consists of $2(c'_i + c''_i + 1)$ metric half-disks, under the expanding operation they turn into $2(c'_i + c''_i + 1) = 2(c_i+1)$ metric half-planes that form the basic domain decomposition for a pole of order $c_i+2$ in the sense of \cite{Boissy}. The boundary loop of $N_i$ corresponds 
to the vanishing cycle around $q_i$ in the shrinking operation, which implies that the resulting pole will be glued to $q_i$ as a node in the limit. In 
addition, the two half-cylinders expand to two half-infinite cylinders, which create two simple poles $q_0$ and $q_{k+1}$ with opposite residues $\pm r$, 
where $r$ encodes the width of the cylinders.  

Let $(R, \eta_R)$ be the limit meromorphic differential out of the expanding operation. We thus conclude that 
$$ (\eta_R) = \left(\sum_{i=1}^k (c_i+ 2)\right)\sigma - \sum_{i=1}^k (c_i+2) q_i - q_{0} - q_{k+1}, $$
and hence the genus of $R$ is zero. Since $q_i = C_i \cap R$ is a separating node, it follows from the global residue condition 
of \cite{BCGGM1} that $\Res_{q_i} \eta_R = 0$. As a cross check, 
$$ \sum_{i=0}^{k+1} \Res_{q_i} \eta_R =  \Res_{q_0} \eta_R + 0 + \cdots + 0 + \Res_{q_{k+1}} \eta_R = 0, $$
hence $\eta_R$ satisfies the Residue Theorem on $R$. Finally, the cylinders are glued to the figure eight boundary on both sides, hence 
the $k+1$ homologous self saddle connections have holonomy equal to $r$ up to sign. Their configuration (holonomy and angles in between) is preserved in the expanding process, hence the limit differential $\eta_R$ possesses the desired configuration. 
\end{proof}

\begin{proposition} 
\label{prop:II-(ii)}
Consider a block of surfaces of type (ii) in $C^{\varepsilon}$, that is, a half-cylinder, followed by $k\geq 0$ surfaces $C_1^{\varepsilon}, \ldots, C_k^{\varepsilon}$, each of genus $g_i \geq 1$ with a figure eight boundary, followed by a surface $C_{k+1}^{\varepsilon}$ of genus $g_{k+1}\geq 1$ 
with a pair of holes boundary.  Let $\sigma$ be the newborn zero of order $\sum_{i=1}^k (c_i+ 2) + (b_{k+1}+1)$. As $\varepsilon \to 0$, we have 
\begin{itemize}
\item The limit differential consists of $k+1$ disjoint surfaces $C_1, \ldots, C_{k+1}$ attached to a component $R\cong \PP^1$ 
at the nodes $q_1, \ldots, q_{k+1}$, respectively. 
\item $R$ contains only $\sigma$ among all the markings. 
\item For each $i = 1, \ldots, k$, $\ord_{q_i} \eta_{C_i} = c_i$ and $\ord_{q_i}\eta_R = - c_i - 2$. 
\item $\ord_{q_{k+1}} \eta_{C_{k+1}} = b_{k+1}$ and $\ord_{q_{k+1}}\eta_R = - b_{k+1} - 2$.  
\item For each $i = 1, \ldots, k$, $\Res_{q_i}\eta_R = 0$. 
\item $\eta_R$ has a simple pole at $q_0  \in R\setminus \{\sigma, q_1, \ldots, q_{k+1} \}$ with 
$\Res_{q_0} \eta_R = - \Res_{q_{k+1}} \eta_R = \pm r$. 
\item $\eta_R$ admits a configuration of type (ii), i.e., it has precisely $k+1$ homologous self saddle connections with angles $2(c'_i+1)\pi$ and $2(c''_i+1)\pi$ in between consecutively for $i=1, \ldots, k$, and with holonomy equal to $r$ up to sign. 
\end{itemize}
\end{proposition}

See Figure~\ref{fig:II-(ii)} for an illustration of the underlying curve of the limit differential. 

\begin{figure}[h]
    \centering
    \psfrag{R}{$R$}
    \psfrag{A}{$C_1$}
    \psfrag{B}{$C_k$}
    \psfrag{C}{$C_{k+1}$}
    \psfrag{s}{$\sigma$}
    \psfrag{0}{$q_0$}
    \psfrag{1}{$q_1$}
    \psfrag{k}{$q_k$}
    \psfrag{l}{$q_{k+1}$}  
    \includegraphics[scale=0.8]{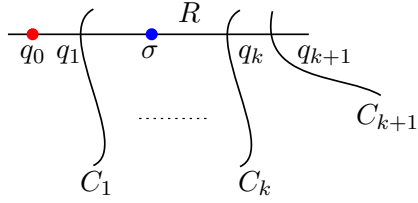}
    \caption{\label{fig:II-(ii)} The underlying curve of the limit differential in Proposition~\ref{prop:II-(ii)}.}
    \end{figure}

\begin{proof}
The proof is almost identical with the preceding one. The only difference occurs at the last surface. A small neighborhood $N_{k+1}$ enclosing 
$\beta_{k+1}$ in $C_{k+1}^{\varepsilon}$ consists of $2(b_{k+1} + 1)$ half-disks, one of which is irregular as in \cite[Figure 8]{EMZ}, 
hence in the expanding process they turn into 
$2(b_{k+1} + 1)$ half-planes, giving a flat geometric neighborhood for a pole of order $b_{k+1} + 2$. Moreover, $N_{k+1}$ is homologous to the $\gamma_i$. The orientation of $N_{k+1}$ is the opposite to that of $N_0$ enclosing the boundary $\alpha_0$ 
of the beginning half cylinder, hence their homology classes add up to zero. We thus conclude that $\Res_{q_0} \eta_R = - \Res_{q_{k+1}} \eta_R$. Alternatively, it follows from the Residue Theorem applied to $R$, since $\Res_{q_i}\eta_R = 0$ for all $i = 1, \ldots, k$. The holonomy of the saddle connections and the angles between them are preserved in the expanding process, hence $\eta_R$ has the configuration as described. 
\end{proof}

\begin{proposition} 
\label{prop:II-(iii)}
Consider a block of surfaces of type (iii) in $C^{\varepsilon}$, that is, a surface $C_{0}^{\varepsilon}$ of genus $g_{k+1}\geq 1$ 
with a pair of holes boundary, followed by $k\geq 0$ surfaces $C_1^{\varepsilon}, \ldots, C_k^{\varepsilon}$, each of genus $g_i \geq 1$ with a figure eight boundary, followed by a surface $C_{k+1}^{\varepsilon}$ of genus $g_{k+1}\geq 1$ 
with a pair of holes boundary.  Let $\sigma$ be the newborn zero of order $\sum_{i=1}^k (c_i+ 2) +  (a_0+1) + (b_{k+1}+1)$. As $\varepsilon \to 0$, we have 
\begin{itemize}
\item The limit differential consists of $k+2$ disjoint surfaces $C_0, \ldots, C_{k+1}$ attached to a component $R\cong \PP^1$ 
at the nodes $q_0, \ldots, q_{k+1}$, respectively. 
\item $R$ contains only $\sigma$ among all the markings. 
\item For each $i = 1, \ldots, k$, $\ord_{q_i} \eta_{C_i} = c_i$ and $\ord_{q_i}\eta_R = - c_i - 2$. 
\item $\ord_{q_{0}} \eta_{C_{0}} = a_0$ and $\ord_{q_{0}}\eta_R = - a_{0} - 2$.  
\item $\ord_{q_{k+1}} \eta_{C_{k+1}} = b_{k+1}$ and $\ord_{q_{k+1}}\eta_R = - b_{k+1} - 2$.  
\item For each $i = 1, \ldots, k$, $\Res_{q_i}\eta_R = 0$. 
\item $\Res_{q_0} \eta_R = - \Res_{q_{k+1}} \eta_R = \pm r$. 
\item $\eta_R$ admits a configuration of type (iii), i.e., it has precisely $k+1$ homologous self saddle connections with angles $2(c'_i+1)\pi$ and $2(c''_i+1)\pi$ in between consecutively for $i=1, \ldots, k$, and with holonomy equal to $r$ up to sign.   
\end{itemize}
\end{proposition}

See Figure~\ref{fig:II-(iii)} for an illustration of the underlying curve of the limit differential. 

\begin{figure}[h]
    \centering
    \psfrag{R}{$R$}
    \psfrag{A}{$C_1$}
    \psfrag{B}{$C_k$}
    \psfrag{C}{$C_{k+1}$}
    \psfrag{D}{$C_{0}$}
    \psfrag{s}{$\sigma$}
    \psfrag{0}{$q_0$}
    \psfrag{1}{$q_1$}
    \psfrag{k}{$q_k$}
    \psfrag{l}{$q_{k+1}$}  
    \includegraphics[scale=0.8]{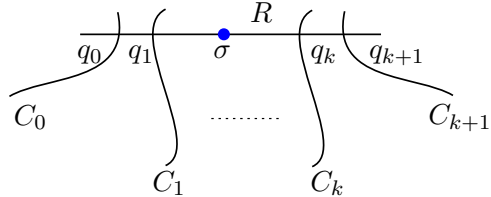}
    \caption{\label{fig:II-(iii)} The underlying curve of the limit differential in Proposition~\ref{prop:II-(iii)}.}
    \end{figure}

\begin{proof}
Since the beginning and ending surfaces both have a pair of holes boundary, the proof follows from the previous two. 
\end{proof}

Let us call the limit twisted differentials in Propositions~\ref{prop:II-(i)}, ~\ref{prop:II-(ii)}, and~\ref{prop:II-(iii)} \emph{surfaces of degenerate type} (i), (ii), and (iii), respectively. In order to glue them to form a global twisted differential, the above proofs (and also the definition of twisted differentials) imply the following gluing pattern. The simple pole $q_0$ (or $q_{k+1}$) in a surface of degenerate type (i) has to be glued with a simple pole in another surface of type (i) or (ii), and the same description holds for $q_0$ in a surface of type (ii). For a surface of type (ii), the component $C_{k+1}$ has to be contained in another surface of type (ii) or (iii). Namely, it has a zero of order $b'_{k+1}$ that is glued with a pole $q'_{k+1}$ of order $b'_{k+1}+2$ in the rational component $R'$  of the other surface. The same description holds for $C_0$ and $C_{k+1}$ in a surface of type (iii). 

\begin{thm}
\label{thm:principal-II}
In the above setting, $\Delta(\vmu, \cC)$ parameterizes twisted differentials constructed by gluing surfaces of degenerate type 
(i), (ii), and (iii).  
\end{thm}

\begin{proof}
Since $C^{\varepsilon}$ admits the configuration $\cC = (L, \{ a_i, b_i\}, \{ c'_j, c''_j \})$, it can be constructed by gluing blocks of surfaces 
of type (i), (ii), and (iii). By applying Propositions~\ref{prop:II-(i)}, ~\ref{prop:II-(ii)}, and~\ref{prop:II-(iii)} simultaneously, we thus conclude that 
the limit twisted differential is formed by gluing surfaces of degenerate type (i), (ii), and (iii) as above.  
\end{proof}

We summarize some useful observation out of the proofs. 

\begin{remark}
If the homologous closed saddle connections in a configuration $\cC$ of type II contains $k$ distinct zeros, then a curve in $\Delta(\vmu, \cC)$ contains 
$k$ rational components. Moreover, if two rational components intersect, then each of them has a simple pole at the node, and the residues at the two branches of the node add up to zero. In general, at the polar nodes the residues are $\pm r$ for a fixed nonzero $r\in \CC$, such that their signs are alternating along 
the (unique) circle in the dual graph of the entire curve, and that the holonomy of the saddle connections is equal to $r$ up to sign. 
\end{remark}

\begin{example}
The limit of the surface in \cite[Figure 7]{EMZ} as the $\gamma_i$ shrink to zero is of the following type: $S_1$, followed by a marked $\PP^1$, followed 
by $S_2$, followed by a marked $\PP^1$, followed by a marked $\PP^1$ with an $S_3$ tail, followed by a marked $\PP^1$ with an $S_4$ tail, and back to $S_1$, see Figure~\ref{fig:EMZ-1}, where $R_1$ is of type (iii), $R_2$ is of type (ii), $R_3$ is of type (i) and $R_4$ is of type (ii). 

\begin{figure}[h]
    \centering
    \psfrag{1}{$S_1$}
    \psfrag{2}{$R_1$}
    \psfrag{3}{$S_2$}
    \psfrag{4}{$R_2$}
    \psfrag{5}{$R_3$}
    \psfrag{6}{$S_3$}
    \psfrag{7}{$S_4$}
    \psfrag{8}{$R_4$}
    \includegraphics[scale=0.8]{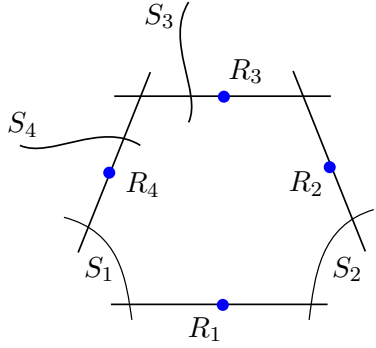}
    \caption{\label{fig:EMZ-1} The underlying curve of the degeneration of \cite[Figure 7]{EMZ}.}
    \end{figure}
\end{example}

\begin{example}
The limit of the surface in \cite[Figure 11]{EMZ} as the $\gamma_i$ shrink to zero is of the following type: a flat torus $E_1$, followed by a chain of two 
$\PP^1$, each with a marked simple zero, followed by a flat torus $E_2$, followed by a chain of two $\PP^1$, each with a marked simple zero, and back to $E_1$, see Figure~\ref{fig:EMZ-2}. Moreover, the differential on each $\PP^1$ has a double pole at the intersection with one of the tori and has a simple pole at the intersection with one of the $\PP^1$. Finally, the residues at the two poles of each $\PP^1$ are $\pm r$ for some fixed nonzero $r\in \CC$, such that 
their signs are alternating along the cyclic dual graph of the entire curve. 

\begin{figure}[h]
    \centering
    \psfrag{1}{$E_1$}
    \psfrag{2}{$E_2$}
      \includegraphics[scale=0.8]{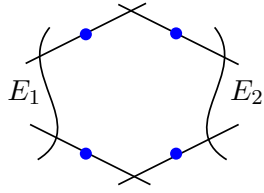}
    \caption{\label{fig:EMZ-2} The underlying curve of the degeneration of \cite[Figure 11]{EMZ}.}
    \end{figure}
\end{example}

%%%%%%%%%%%%%%%%%%%%%%%%%%%%%%
\subsection{Meromorphic differentials of type II on $\PP^1$}
%%%%%%%%%%%%%%%%%%%%%%%%%%%%%%

Recall in Proposition~\ref{prop:unique-moduli-typeI} we showed that differentials on $\PP^1$ admitting a given configuration of type I are unique up to scale. The same result holds for differentials on $\PP^1$ admitting a given configuration of type (i), (ii), or (iii) as above. 

\begin{proposition}
\label{prop:unique-moduli-typeII}
Let $\eta_0$ be a differential on $\PP^1$ that admits a configuration of type either (i), (ii), or (iii) as described in Propositions~\ref{prop:II-(i)}, ~\ref{prop:II-(ii)}, and~\ref{prop:II-(iii)}. Then up to scale such $\eta_0$ is unique. 
\end{proposition}

\begin{proof}
We provide a constructive proof for the case of type (i), which is analogous to the proof of 
Proposition~\ref{prop:unique-moduli-typeI}. The other two types follow similarly. 

Let us make some observation first. Suppose $\eta_0$ is a differential on $\PP^1$ with a unique zero $\sigma$ and $k+2$ poles $q_0, \ldots, q_{k+1}$ such that $\Res_{q_i} \eta_0 = 0$ 
for $i = 1, \ldots, k$, and that $\Res_{q_0} \eta_0 = - \Res_{q_{k+1}} = \pm r$ for a nonzero $r$. Let $\alpha$ and $\beta$ be two self saddle connections of $\eta_0$. 
Treat them as closed loops in $\CC = \PP^1 \setminus \{q_{k+1} \}$. Then the indices of $\alpha$ and $\beta$ to $q_0$ cannot be zero, for otherwise the integral of $\eta_0$ along them would be zero, contradicting that they are saddle connections of positive length. Therefore, both of them enclose $q_0$ in $\CC$, hence by the Residue 
Theorem 
$$ \int_{\alpha} \eta_0 = \int_{\beta} \eta_0 = \pm r. $$ 
We conclude that in this case all saddle connections of $\eta_0$ are homologous with holonomy equal to $\pm r$. 

Now suppose $\eta_0$ admits the configuration of type (i) (as the description for $\eta_R$ in Proposition~\ref{prop:II-(i)}). 
Rescale $\eta_0$ such that the holonomy of the saddle connections $\gamma_1, \ldots, \gamma_{k+1}$ is $1$. By the preceding paragraph, $\eta_0$ has no other saddle connections. Cut the flat surface $\eta_0$ along all horizontal directions through the unique zero $\sigma$. Since $\eta_0$ has two simple poles with opposite residues equal to $\pm 1$, we see two half-infinite cylinders with boundary given by the first and the last saddle connections $\gamma_1$ and $\gamma_{k+1}$, respectively. The rest part of $\eta_0$ splits into half-planes as basic domains in the sense of \cite{Boissy}, which are of two types according to their boundary. The boundary of the half-planes of the first type contains $\sigma$ that emanates two half-lines to infinity on both sides. The boundary of the half-planes of the second type, from left to right, consists of a half-line ending at $\sigma$, followed by a saddle connection $\gamma_i$, and then a half-line emanated from $\sigma$.  

Since the angles between $\gamma_i$ and $\gamma_{i+1}$ are given on both sides inside the open surface (after removing the two half-infinite cylinders), this configuration determines how these half-planes are glued together. More precisely, say in the counterclockwise direction the angle between $\gamma_i$ and $\gamma_{i+1}$ is $2\pi (c'_i+1)$. Then starting from the upper half-plane $S_i^{+}$ of the second type containing $\gamma_i$ in the boundary and turning counterclockwise, we will see $c'_i$ pairs of lower and upper half-planes of the first type, and then the lower half-plane $S_{i+1}^{-}$ of the second type containing 
$\gamma_{i+1}$ in the boundary. Repeat this process for each $i$ on both sides. We conclude that the gluing pattern of these half-planes 
is uniquely determined by the configuration. After gluing, the resulting open surface has a single figure eight boundary formed by $\gamma_1$ and $\gamma_{k+1}$ at the beginning and at the end, which is then identified with the boundary of the two half-infinite cylinders to recover $\eta_0$.   
Finally, since the angles between $\gamma_i$ and $\gamma_{i+1}$ are $2\pi (c'_i+1)$ and $2\pi (c''_i+1)$ on both sides, 
it determines precisely $c'_i + c''_i + 1 = c_i + 1$ paris of upper and lower half-planes that share the same point at infinity. In other words, they give rise to a flat geometric representation of a pole of order $c_i + 2$, which is the desired pole order for $i = 1, \ldots, k$. 
\end{proof}

%%%%%%%%%%%%%%%%%%%%%%%%%%%%
%%%%%%%%%%%%%%%%%%%%%%%%%%%%
\section{Spin and hyperelliptic structures}
\label{sec:parity}
%%%%%%%%%%%%%%%%%%%%%%%%%%%%
%%%%%%%%%%%%%%%%%%%%%%%%%%%%

For special $\vmu$, the stratum $\cH(\vmu)$ can be disconnected. Kontsevich and Zorich (\cite{KZ}) classified 
connected components of $\cH(\vmu)$ for all $\vmu$. Their result says that $\cH(\vmu)$ can have up to three connected components, where the extra components are caused by spin and  hyperelliptic structures. 

%%%%%%%%%%%%%%%%%%%%%%%%
\subsection{Spin structures}
%%%%%%%%%%%%%%%%%%%%%%%%

We first recall the definition of spin structures. Suppose $\vmu = (2k_1, \ldots, 2k_n)$ is a partition of $2g-2$ with even entries only. For an abelian differential $(C, \omega) \in \cH(\vmu)$, let 
$$ (\omega) = 2k_1 \sigma_1 + \cdots + 2k_n \sigma_n $$
be the associated canonical divisor. Then the line bundle 
$$ \cL = \cO (k_1\sigma_1 + \cdots + k_n \sigma_n) $$ 
is a square root of the canonical line bundle, hence $\cL$ gives rise to a spin structure (also called a theta characteristic). Denote by 
$$ h^0(C, \cL) \pmod{2} $$
the parity of $\omega$. By Atiyah (\cite{Atiyah}) and Mumford (\cite{Mumford}), parities of theta characteristics are deformation invariant. We also refer to $\omega$ along with its parity as a spin structure, which can be either even or odd, and denote the parity by 
$\phi(\omega)$. 

Alternatively, there is a topological description for spin structures using the Arf invariant, due to Johnson ([J]). For a smooth simple closed curve $\alpha$ on a flat surface, let $\ind (\alpha)$ be the degree of the Gauss map from $\alpha$ to the unit circle. Namely, $2\pi \cdot \ind(\alpha)$ is the total change of the angle of the unit tangent vector to $\alpha$ under the flat metric as it moves along $\alpha$ one time. 

Let $\{a_i, b_i \}_{i=1}^g$ 
be a symplectic basis of $C$, i.e., $a_i \cdot a_j = b_i \cdot b_j = 0$ and $a_i \cdot b_j = \delta_{ij}$ for $1\leq i, j \leq g$. 
When $\omega$ has only even zeros, the parity $\phi(\omega)$ can be equivalently defined as 
$$ \phi(\omega) = \sum_{i=1}^g (\ind (a_i) + 1) (\ind (b_i) + 1) \pmod{2}. $$
Suppose we change the choice of the bases, say, by letting $a_i$ cross a zero $\sigma_j$ from one side to the other. Since the zero order of $\sigma_j$ is even, $\ind (a_i)$ remains unchanged mod $2$, hence $\phi(\omega)$ is independent of the choice of the symplectic bases. 

%%%%%%%%%%%%%%%%%%%%%%%%
\subsection{Hyperelliptic structures}
%%%%%%%%%%%%%%%%%%%%%%%%

Next we recall the definition of hyperelliptic structures. There are two cases: $\vmu = (2g-2)$ and $\vmu = (g-1, g-1)$. For 
$(C, \omega)\in \cH(2g-2)$, if $C$ is hyperelliptic and $\tau^{-1}\omega = - \omega$, where $\tau$ is the hyperelliptic involution of $C$, then we say that $(C, \omega)$ has a hyperelliptic structure. Equivalently in this case, the unique zero $\sigma$ of $\omega$ is a Weierstrass point, i.e., $\sigma$ is a ramification point of the hyperelliptic double cover $C\to \PP^1$. 

For $(C, \omega) \in \cH(g-1, g-1)$, similarly if $C$ is hyperelliptic and $\tau^{-1}\omega = - \omega$, then we say that $(C, \omega)$ has a hyperelliptic structure. Equivalently in this case, the two zeros $\sigma_1$ and $\sigma_2$ of $\omega$ are hyperelliptic conjugates of each other, i.e., $\sigma_1$ and $\sigma_2$ have the same image under the hyperelliptic double cover. 

From the viewpoint of flat geometry, we remark that $-\omega$ means rotating the flat surface corresponding to $\omega$ by $180$ degree. Moreover, a hyperelliptic structure in general requires more than that the underlying curve $C$ is hyperelliptic. 

%%%%%%%%%%%%%%%%%%%%%%%%
\subsection{Connected components of  $\cH(\vmu)$}
%%%%%%%%%%%%%%%%%%%%%%%%

Now we can state precisely the classification of connected components of $\cH(\vmu)$ in \cite{KZ}: 
\begin{itemize}
\item Suppose $g\geq 4$. Then 
\subitem $\cH(2g-2)$ has three connected components: the hyperelliptic component $\cH^{\hyp}(2g-2)$, the 
odd spin component $\cH^{\odd}(2g-2)$, and the even spin component $\cH^{\even}(2g-2)$. 
\subitem $\cH(g-1, g-1)$, when $g$ is odd, has three connected components: the hyperelliptic component $\cH^{\hyp}(g-1,g-1)$, the 
odd spin component $\cH^{\odd}(g-1,g-1)$, and the even spin component $\cH^{\even}(g-1,g-1)$. 
\subitem $\cH(g-1, g-1)$, when $g$ is even, has two connected components:  the hyperelliptic component $\cH^{\hyp}(g-1,g-1)$ and the nonhyperelliptic component $\cH^{\nonhyp}(g-1, g-1)$. 
\subitem All the other strata of the form $\cH(2k_1, \ldots, 2k_n)$ have two connected components: the 
odd spin component $\cH^{\odd}(2k_1, \ldots, 2k_n)$ and the even spin component $\cH^{\even}(2k_1, \ldots, 2k_n)$.
\subitem All the remaining strata are connected. 

\item Suppose $g=3$. Then 
\subitem $\cH(4)$ has two connected components: the hyperelliptic component $\cH^{\hyp}(4)$ and the 
odd spin component $\cH^{\odd}(4)$, where the even spin component coincides with the hyperelliptic component.
\subitem  $\cH(2,2)$ has two connected components: the hyperelliptic component $\cH^{\hyp}(2,2)$ and the 
odd spin component $\cH^{\odd}(2,2)$, where the even spin component coincides with the hyperelliptic component.
\subitem All the other strata are connected. 

\item Suppose $g=2$. Then both $\cH(2)$ and $\cH(1,1)$ are connected. Each of them coincides with its hyperelliptic component. 
\end{itemize}

%%%%%%%%%%%%%%%%%%%%%%%%
\subsection{Degeneration of spin structures}
%%%%%%%%%%%%%%%%%%%%%%%%

Let $\cS_g$ be the moduli space of spin structures on smooth genus $g$ curves. The natural morphism $\cS_g \to \cM_g$ is an unramified cover of degree $2^{2g}$. Moreover, $\cS_g$ is a disjoint union of $\cS_g^{+}$ and $\cS_g^{-}$, parameterizing even and odd spin structures, respectively. Cornalba (\cite{Cornalba}) constructed a compactified moduli space of spin structures $\ocS_g = \ocS_g^{+} \sqcup \ocS_g^{-}$ over $\ocM_g$, whose boundary parameterizes degenerate spin structures on stable nodal curves and distinguishes their parities. 
 
We first recall spin structures on nodal curves of compact type. Suppose a nodal curve $C$ consists of $k$ irreducible components $C_1, \ldots, C_k$ such that each of the nodes is separating, i.e., removing it disconnects $C$. Let $L_i$ be a theta characteristic 
on $C_i$, i.e., $L_i^{\otimes 2} = K_{C_i}$. At each node of $C$, insert a $\bbP^1$-bridge, called an exceptional component, and take the line bundle 
$\cO(1)$ on it. Then the collection $\{ (C_i, L_i) \}_{i=1}^k$ along with $\cO(1)$ on each exceptional component 
gives a spin structure on $C$, whose parity is determined by 
$$ h^0(C_1, L_1) + \cdots + h^0(C_k, L_k) \pmod{2}. $$
In particular, if $C_i$ has genus $g_i$, then $g_1 + \cdots + g_k = g$. On each $C_i$ there are $2^{2g_i}$ distinct theta characteristics, hence in total they glue to $2^{2g}$ spin structures on $C$, which equals the number of theta characteristics on a smooth curve of genus $g$. 

If $C$ is not of compact type, the situation is more complicated, because there are two types of spin structures. For example, consider the case 
when $C$ is an irreducible one-nodal curve, by identifying two points $q_1$ and $q_2$ in its normalization $C'$ as a node $q$. For the first type, one can take a square root $L$ of the dualizing line bundle $\omega_C$, which gives $2^{2g-1}$ such spin structures. Equivalently, pull back $L$ 
to $L'$ on $C'$. Then $L'$ is a square root of $K_{C'}(q_1 + q_2)$, and there are $2^{2g-2}$ such $L'$ on $C'$. By Riemann-Roch, 
$h^0(C', L') - h^0(C', L'(-q_1 - q_2)) = 1$, hence neither $q_1$ nor $q_2$ is a base point of $L'$, and any section $s$ of $L'$ that vanishes 
at one of the $q_i$ must also vanish at the other. Therefore, the space of sections $H^0(C', L')$ has a decomposition 
$V_0 \oplus \langle s \rangle$, where $V_0$ 
is the subspace of sections that vanish at $q_1$ and $q_2$, and $s$ is a section not vanishing at the $q_i$. Note that $L^{\otimes 2} = \omega_C$, whose fibers over $q_1$ and $q_2$ have a canonical identification by $\Res_{q_1} \omega + \Res_{q_2} \omega = 0$, where $\omega$ is a stable differential with at worst simple poles at the $q_i$, treated as a local section of $\omega_C$ at $q$. It implies that in order to glue the fibers of $L'$ over $q_1$ and $q_2$ to form $L$ on $C$, there are two choices, and exactly one of the two preserves $s$ as a section of $L$. We thus conclude that this way gives $2^{2g-1}$ spin structures on $C$, where half of them are even and the other half are odd. For the second type, insert an exceptional 
$\bbP^1$-component connecting $q_1$ and $q_2$ in $C'$. Take an ordinary theta characteristic $L'$ on $C'$ and the bundle $\cO(1)$ on $\bbP^1$. 
In this way one obtains $2^{2g-2}$ such $L'$. For a fixed $L'$, there is no extra choice of gluing $L'$ to $\cO(1)$ at $q_1$ and $q_2$, due to the automorphisms of $\cO(1)$ on $\bbP^1$, and hence the parity of the resulting spin structure equals that of $\eta'$. Nevertheless, the morphism $\ocS_g \to \ocM_g$ is simply ramified along the locus of such $\eta'$ of the second type. Therefore, taking both types into account along with the multiplicity factor for the second type, we again obtain the number $2^{2g}$, which is equal to the degree of $\ocS_g \to \ocM_g$. 

Below we describe a relation between degenerate spin structures and twisted differentials. Suppose a twisted differential $(C, \eta)$ is in the closure  of 
a stratum $\cH(\vmu)$ that contains a spin component, i.e., when $\vmu$ has even entries only. For a node $q$ joining two components $C_1$ and $C_2$ of $C$, 
by definition $\ord_q \eta_1 + \ord_q \eta_2 = -2$. If both orders are odd, we do nothing at $q$. If both orders are even, we insert an exceptional $\PP^1$ 
at $q$. In particular if $q$ is separating, in this case $\ord_q \eta_1$ and $\ord_q \eta_2$ are both even, because each side of $q$ contains even zeros only, and hence we insert a $\PP^1$ at $q$, which matches the preceding discussion on curves of compact type. Now suppose $\eta_i$ on a component 
$C_i$ of $C$ satisfies that 
$$ (\eta_i) = \sum_j 2m_j z_j  + \sum_k 2n_k q_k + \sum_l (2h_l - 1) q_l, $$
where the $z_j$ are the zeros in the interior of $C_i$, the $q_k$ are the nodes of even order in $C_i$, and the $q_l$ are the nodes of odd order in $C_i$. Consider the bundle 
$$ L_i = \cO\left(\sum_j m_j z_j + \sum_k n_k q_k + \sum_l h_l q_l\right) $$
on $C_i$. Then the collection $(C_i, L_i)$ along with the exceptional components and $\cO(1)$ gives a spin structure $\cL$ on $C$. However, if 
$(C, \eta)$ has a node of odd order, i.e., a node without inserting an exceptional component, then there are two gluing choices at such a node, as described above, hence $\cL$ is only determined by $(C, \eta)$ up to finitely many choices, and its parity may vary with different choices. From the viewpoint of smoothing twisted differentials, it means that different choices of opening up nodes of $C$ may deform $(C, \eta)$ into different connected components of $\cH(\vmu)$. 

The idea behind the above description is as follows. For a node $q$ joining two components $C_1$ and $C_2$, if there is no twist at $q$, i.e., if 
$\ord_q \eta_1 = \ord_q \eta_2 = -1$, then locally at $q$ one can directly take a square root of $\omega_C$. If $\ord_q \eta_1$ and $\ord_q \eta_2$ are both odd, i.e., if the twisting parameter $\ord_{q}\eta_i - (-1)$ is even, then its one-half gives the twisting parameter for the limit spin bundle on $C$. On the other hand if $\ord_q \eta_1$ and $\ord_q \eta_2$ are even, then the twisting parameter $\ord_{q}\eta_i - (-1)$ is not divisible by $2$, hence one has to insert an exceptional $\PP^1$ at $q$, which is twisted once to make the twisting parameters at the new nodes even. As a consequence, the resulting twisted differential restricted to $\PP^1$ is $\cO(2)$, hence its one-half is the bundle $\cO(1)$ encoded in the degenerate spin structure. The reader may refer to \cite{FP} for a detailed explanation. 

%%%%%%%%%%%%%%%%%%%%%%%%
\subsection{Degeneration of hyperelliptic structures}
%%%%%%%%%%%%%%%%%%%%%%%%

Next we describe how hyperelliptic structures degenerate. Recall that the closure of the locus of hyperelliptic curves of genus $g$ 
in $\ocM_g$ can be identified with the moduli space $\tcM_{0,2g+2}$ parameterizing stable rational curves with $2g+2$ unordered markings, where 
the markings correspond to the $2g+2$ branch points of hyperelliptic covers. On the boundary of the moduli spaces, hyperelliptic covers degenerate to admissible double covers of stable genus zero curves in the setting of Harris-Mumford (\cite{HM}). Therefore, Weierstrass points on smooth hyperelliptic curves degenerate to ramification points in such admissible hyperelliptic covers, and the limits of a pair of hyperelliptic conjugate points remain to be conjugate in the limit admissible cover, see Figure~\ref{fig:adm}. 

\begin{figure}[h]
    \centering
    \includegraphics[scale=0.8]{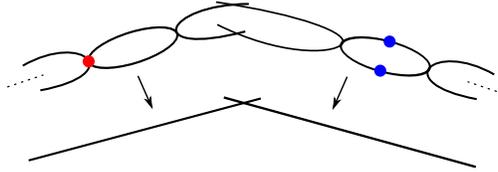}
    \caption{\label{fig:adm} A limit of Weierstrass points (labeled by red) and a limit of pairs of conjugate points (labeled by blue) in a hyperelliptic admissible double cover.}
    \end{figure}

%%%%%%%%%%%%%%%%%%%%%%%%
\subsection{Spin and hyperelliptic structures for the principal boundary of type I}
\label{subsec:parity-I}
%%%%%%%%%%%%%%%%%%%%%%%%

Let $\cC =  (m_1, m_2, \{a'_i, a''_i\}_{i=1}^p)$ be an admissible configuration of type I for a stratum $\cH(\vmu)$. 
Suppose $(C, \eta)$ is a twisted differential contained in $\Delta(\vmu, \cC)$. By the description of $\Delta(\vmu, \cC)$ in 
Section~\ref{ss:principal-boundaryI}, $C$ consists of $p$ components $C_1, \ldots, C_p$, each of genus $g_i \geq 1$ 
with $g_1 + \cdots + g_p = g$, attached to a rational component $R$, and $\eta_i$ is the differential of $\eta$ restricted to $C_i$ satisfying that $(\eta_i) = (2g_i - 2) q_i$, where $q_i$ is the node joining $C_i$ with $R$. 

Consider the case when $\vmu$ has even entries only. Then $\cH(\vmu)$ contains an even spin component and an odd spin component (and possibly a hyperelliptic component). This parity distinction can be extended to the principal boundary $\Delta(\vmu, \cC)$, see \cite[Lemma 10.1]{EMZ} for a proof using the Arf invariant. For the reader's convenience, below we recap the result and also provide an algebraic proof.

\begin{proposition}
Let $(C, \eta)$ be a twisted differential in $\Delta(\vmu, \cC)$ described as above, with even zeros only. Then 
the parity of $\eta$ is 
$$ \phi(\eta) = \phi(\eta_1) + \cdots + \phi(\eta_p) \pmod{2}. $$
\end{proposition}

\begin{proof}
Since $(\eta_i) = (2g_i - 2) q_i$, the degenerate spin structure on 
$C_i$ is given by $\cO((g_i-1) q_i)$ in the sense of Cornalba (\cite{Cornalba}). Moreover, on the rational component $R$, any theta characteristic 
has even parity (given by zero). Since $C$ is of compact type, the parity of $\eta$ is equal to the sum of the parities of the $\eta_i$, as claimed.  
\end{proof}

\begin{corollary}
Suppose $\cC$ is of type $I$ and $\vmu$ contains only even zeros. Then differentials in the thick part of $\cH(\vmu)$ degenerate to twisted differentials in $\Delta(\vmu, \cC)$ with the same parity. 
\end{corollary}

Note that for the parity discussion we only require that $a'_i + a''_i$ is even for each $i$, and there is no other requirement for the individual values of $a'_i$ and $a''_i$. 

Next we consider hyperelliptic components. Since configurations of type I require at least two distinct zeros, here we only need to treat the case $\vmu = (g-1, g-1)$, which contains a hyperelliptic component $\cH^{\hyp}(g-1, g-1)$ (and possibly spin components if $g$ is odd). 

The following result is a reformulation of \cite[Lemma 10.3]{EMZ}. Here we again provide an algebraic proof. 

\begin{proposition}
\label{prop:hyp-I}
Suppose $(C, \eta)$ is a twisted differential contained in $\Delta(g-1, g-1, \cC)$. Then differentials 
in the thick part of $\cH^{\hyp}(g-1, g-1)$ can degenerate to $(C, \eta)$ if and only if either 
\begin{itemize}
\item $p = 1$, $(C_1, \eta_1) \in \cH^{\hyp}(2g-2)$, $a'_1 = a''_1 = g-1$, or 
\item $p = 2$, $ (C_i, \eta_i) \in \cH^{\hyp}(2g_i-2)$, $a'_i = a''_i = g_i - 1$ for $i= 1, 2$. 
\end{itemize}
\end{proposition}

\begin{proof}
Suppose $(C, \eta)$ is a degeneration of differentials from $\cH^{\hyp}(g-1, g-1)$. Then $C$ admits an admissible hyperelliptic double cover $\pi$, where the two zeros $\sigma_1$ and $\sigma_2$ are conjugates under $\pi$. Since each $C_i$ meets the rational component $R$ at a single node $q_i$, and $C_i$ is not rational, by the definition of admissible covers, $q_i$ has to be a ramified node under $\pi$. By the Riemann-Hurwitz formula, $\pi$ restricted to $R$ has only two ramification points, which implies that $p \leq 2$. 

For $p=1$, $C_1$ has genus $g$, and it admits a hyperelliptic double cover with $q_1$ being a ramification point, hence $(C_1, \eta_1) \in \cH^{\hyp}(2g-2)$. Moreover, there is only one saddle connection joining $\sigma_1$ to $\sigma_2$, so the angle condition in the configuration $\cC$ can only be $a'_1 = a''_1 = g-1$. See Figure~\ref{fig:adm-I-1} for this case and the corresponding hyperelliptic admissible cover.  

\begin{figure}[h]
    \centering
      \psfrag{1}{$\sigma_1$}
     \psfrag{2}{$\sigma_2$}
     \psfrag{3}{$q_1$}
      \psfrag{R}{$R$}
       \psfrag{C}{$C_1$}
    \includegraphics[scale=0.8]{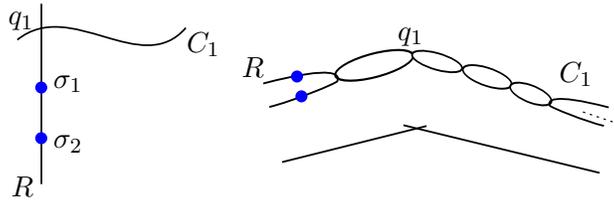}
    \caption{\label{fig:adm-I-1} The case $p=1$ in Proposition~\ref{prop:hyp-I} and the corresponding hyperelliptic admissible cover.}
      \end{figure}

For $p=2$, by the same argument as above we see that $ (C_i, \eta_i) \in \cH^{\hyp}(2g_i-2)$ for $i = 1, 2$. In addition, since the hyperelliptic involution interchanges $\sigma_1$ and $\sigma_2$, it also swaps the two saddle connections $\gamma_1$ and $\gamma_2$ (even on the degenerate component $R$). It follows that $a'_i = a''_i$ for $i = 1,2$. Since $a'_i + a''_i = 2g_i - 2$, we thus conclude that $a'_i = a''_i = g_i - 1$. See Figure~\ref{fig:adm-I-2} for this case and the corresponding hyperelliptic admissible cover.  

\begin{figure}[h]
    \centering
      \psfrag{1}{$C_1$}
     \psfrag{2}{$C_2$}
     \psfrag{r}{$\sigma_1$}
      \psfrag{s}{$\sigma_2$}
       \psfrag{a}{$q_1$}
        \psfrag{b}{$q_2$}
           \psfrag{R}{$R$}
       \includegraphics[scale=0.8]{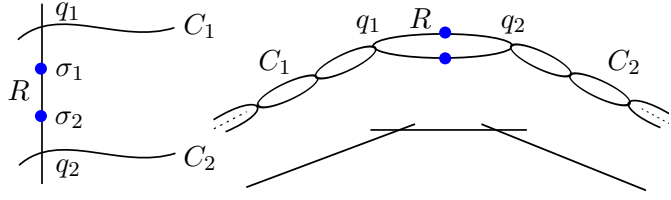}
    \caption{\label{fig:adm-I-2} The case $p=2$ in Proposition~\ref{prop:hyp-I} and the corresponding hyperelliptic admissible cover.}
      \end{figure}

Conversely if $(C, \eta)$ belongs to one of the two cases, the smoothing operation in the proof of Theorem~\ref{thm:principal-I} implies that nearby flat surfaces after opening up the nodes are contained in $\cH^{\hyp}(g-1, g-1)$.
\end{proof}

Denote by $\Delta^{\hyp}(\cdot)$, $\Delta^{\even}(\cdot)$, and $\Delta^{\odd}(\cdot)$ the respective loci of twisted differentials in the principal boundary that are degenerations from hyperelliptic and spin components as specified in the above propositions. We summarize our discussion as follows. 

\begin{corollary}
Let $\cC$ be an admissible configuration of type I for $\cH(\vmu)$. Then the principal boundary $\Delta(\vmu, \cC)$ satisfies the following description: 
\begin{itemize}
\item Suppose $g$ is odd. 
\subitem For $\cC = (m_1 = m_2 = g-1, p = 1, a'_1 = a''_1 = g-1)$ or $\cC = (m_1 = m_2 = g-1, p = 2, a'_i = a''_i = g_i - 1)$ with $g_1 + g_2 = g$, $\Delta(g-1, g-1, \cC)$ is a disjoint union of $\Delta^{\hyp}(g-1, g-1, \cC)$, $\Delta^{\odd}(g-1, g-1, \cC)$, and 
$\Delta^{\even}(g-1, g-1, \cC)$. 
\subitem For all the other types $\cC$, $\Delta(g-1, g-1, \cC)$ is a disjoint union of $\Delta^{\odd}(g-1, g-1, \cC)$ and $\Delta^{\even}(g-1, g-1, \cC)$. 

\item Suppose $g$ even. 
\subitem For $\cC = (m_1 = m_2 = g-1, p = 1, a'_1 = a''_1 = g-1)$ or $\cC = (m_1 = m_2 = g-1, p = 2, a'_i = a''_i = g_i - 1)$ with $g_1 + g_2 = g$, $\Delta(g-1, g-1, \cC)$ is a disjoint union of $\Delta^{\hyp}(g-1, g-1, \cC)$ and  $\Delta^{\nonhyp}(g-1, g-1, \cC)$. 
\subitem For all the other types $\cC$, $\Delta(g-1, g-1, \cC)$ coincides with $\Delta^{\nonhyp}(g-1, g-1, \cC)$. 

\item For all the remaining types $\cC$ and $\vmu$ with even entries only,  $\Delta(\vmu, \cC)$ is a disjoin union of 
$\Delta^{\odd}(\vmu, \cC)$ and $\Delta^{\even}(\vmu, \cC)$. 
\end{itemize}
\end{corollary}

\begin{remark}
In the above corollary, each $\Delta^{\hyp}(\cdot)$, $\Delta^{\even}(\cdot)$, or $\Delta^{\odd}(\cdot)$ can be disconnected, since in general they are unions of products of strata in lower genera. Moreover for small $g$, some of them can also be empty. 
\end{remark}

%%%%%%%%%%%%%%%%%%%%%%%%
\subsection{Spin and hyperelliptic structures for the principal boundary of type II}
\label{subsec:parity-II}
%%%%%%%%%%%%%%%%%%%%%%%%

Let $\cC = (L, \{a_i, b_i \}, \{c'_j, c''_j \})$ be a configuration of type II for a stratum $\cH(\vmu)$. Consider the case when $\vmu$ has even entries only, i.e., a differential in $\cH(\vmu)$ has odd or even parity. The parity distinction can be extended to the principal boundary $\Delta(\vmu, \cC)$, see \cite[Section 14.1]{EMZ}. Below we recap the results and also provide alternative algebraic proofs. 

Recall the description for $(C, \eta)$ in Theorem~\ref{thm:principal-II}. Let us first simplify the statement of \cite[Lemma 14.1]{EMZ} in our setting. 

\begin{lemma}
\label{lem:spin-type}
Let $(C, \eta)$ be a twisted differential contained in $\Delta(\vmu, \cC)$. 
Suppose $\vmu$ has even zeros only. Then the following conditions hold: 
\begin{itemize}
\item $\eta$ has even zero order at each marking of $C$. 

\item $\eta$ has even zero and pole order at a separating node of $C$. 

\item For all non-separating nodes of $C$, the zero and pole orders of $\eta$ are either all even, or all odd. 
\end{itemize}
\end{lemma}

\begin{proof}
Because $\vmu$ has even zeros only, and those zeros are the markings of $C$, the first condition holds by definition of twisted differentials. 

Suppose $q$ is a separating node of $C$. By the description of $C$ in Theorem~\ref{thm:principal-II}, $q$ joins a component $C_i$ with 
a rational component $R$. Since the markings in the interior of $C_i$ are even zeros, we conclude that $\ord_q \eta_{C_i}$ has the same parity as $2g_{C_i}-2$, hence it is even, which implies the second condition. 

Finally, recall that all non-separating nodes bound the (unique) cycle in the dual graph of $C$. Since $\eta$ has even order at all the other nodes and at all markings, going along the edges of the cycle one by one, the parity of the order of $\eta$ at one vertex of the cycle determines 
that all the others have the same parity, hence the last condition holds. 
\end{proof}

\begin{remark}
If $\eta$ has even order at all non-separating nodes, then there is no rational component $R$ in the central cycle of $C$ that has a simple polar node. In that case types (i) and (ii) do not appear 
in the description of $C$, which is exactly the way \cite[Lemma 14.1]{EMZ} phrased. 
\end{remark}

Next, we interpret \cite[Lemmas~14.2, 14.3, and 14.4]{EMZ} in terms of Cornalba's spin structures. 

\begin{lemma}
\label{lem:spin-(i)}
Suppose all rational components of $(C, \eta)$ are of type (i). Then the limit spin structure on $(C, \eta)$ has parity 
$$ \phi (C, \eta) = \sum_{i=1}^p \phi (C_i, \eta_{C_i}) + \sum (c'_i + 1) + 1. $$
\end{lemma}

\begin{proof}
We first remark that since $\eta$ has even zeros only, $c'_i + c''_i$ is even for all $i$, hence using $c'_i$ or $c''_i$ does not matter for the parity formula. 

Next, since only type (i) appears in the description of $C$, each $C_i$ is a tail of $C$, which is attached to $C$ at a separating node, hence the limit spin structure on $C_i$ is generated by one-half of $(\eta_{C_i})$, and it contributes $\phi(C_i, \eta_{C_i})$ to the total parity. 

The central cycle $S$ of $C$ is a loop of rational components $R_1, \ldots, R_k$ in a cyclic order. At each node $q_i$ joining 
$R_i$ to $R_{i+1}$, $\eta$ has a simple pole on the two branches of $q_i$ with opposite residues $\pm r$, hence in the limit spin structure we preserve $q_i$ and do not insert an exceptional $\PP^1$ component. Therefore, the limit spin structure restricted to $S$ is a square root $L$ of $\omega_S$, where $S$ has arithmetic genus one, and $L|_{R_i} = \cO_{R_i}$. Starting from $R_1$, identify the fibers of $\cO_{R_1}$ 
and $\cO_{R_2}$ at $q_1$, then identify the fibers of $\cO_{R_2}$ and $\cO_{R_3}$ at $q_2$, so on and so forth. The last identification between the fibers of $\cO_{R_k}$ and $\cO_{R_1}$ at $q_k$ has two choices, which makes $h^0(S, L) = 0$ or $1$. Hence the parity of the spin structure on $S$ varies with the gluing choice, where the gluing choice is actually determined by the configuration data $\{ c'_i, c''_i\}$. By analyzing the Arf invariant, the parity contribution from $S$ is $\sum (c'_i + 1) + 1$, see the proof of \cite[Lemma 14.2]{EMZ} for details. 
\end{proof} 

Now we consider the last alternate conditions in Lemma~\ref{lem:spin-type}. 

\begin{lemma}
\label{lem:spin-even}
Suppose $\eta$ has even order at all non-separating nodes of $C$. Then the parity of the limit spin structure on $(C, \eta)$ is 
$$ \phi (C, \eta) = \sum_{i=1}^p \phi (C_i, \eta_{C_i}), $$
where the $C_i$ are the non-rational components of $C$. 
\end{lemma}

\begin{proof}
In this case on each $C_i$ the limit spin structure is generated by one-half of $(\eta_{C_i})$, because $\eta_{C_i}$ has even zeros at the markings and nodes. The rational component $R_i$ joining 
$C_i$ and $C_{i+1}$ plays the role of an exceptional component in the limit spin structure, and carries the bundle $\cO(1)$. Therefore, the total parity 
is the sum of the parities over all $C_i$. 
\end{proof}

\begin{lemma}
\label{lem:spin-odd}
Suppose $\eta$ has odd order at every non-separating node of $C$. Let $N$ be the total number of nearby flat surfaces under the previous smoothing procedure. Then exactly 
$N/2$ of them have odd spin structure and $N/2$ have even spin structure. 
\end{lemma}

\begin{proof}
Let $S$ be the central cycle of $C$. Then $\eta$ has odd zeros and poles at all the nodes of $S$. Hence in the limit spin structure we do not insert an exceptional component at each node of $S$. Therefore, given the spin structure on each component of $S$, we have different gluing choices to form a global spin structure on $C$. When varying the gluing choice over one node of $S$ while keeping the others, the parity of the resulting spin structure differs by one, hence the desired claim follows. See also the proof of \cite[Lemma 14.4]{EMZ} for an argument using the Arf invariant. 
\end{proof}

Next we consider the principal boundary of type II for hyperelliptic components. Below we recap \cite[Lemmas 14.5 and 14.6]{EMZ} and provide algebraic proofs using hyperelliptic admissible covers. 

\begin{lemma}
\label{lem:hyp-1}
Suppose $(C, \eta)$ is in the principal boundary $\Delta(2g-2, \cC)$ for a configuration $\cC$ of type II. Then $(C, \eta)$ is in $\Delta^{\hyp}(2g-2, \cC)$ if and only if it is one of the following types: 
\begin{itemize}
\item[{\rm(1)}] $C$ has two components $C_1$ and $R$ meeting at two nodes $q_1$ and $q_2$, $(C_1, \eta_{C_1}) \in \cH^{\hyp}(g-2,g-2)$ with $q_1$ and $q_2$ as the two zeros, and 
$R$ contains the unique marking $\sigma$. 

\item[{\rm (2)}] $C$ has two components $C_1$ and $E$ meeting at one node $\sigma'$, where $E$ is an irreducible one-nodal curve by identifying two points $q_1$ and $q_2$ in $R$, 
$(C_1, \eta_{C_1}) \in \cH^{\hyp}(2g-4)$ with $\sigma'$ as the zero, and $E$ contains the unique marking $\sigma$. 

\item[{\rm (3)}] $C$ has three components $C_1$, $C_2$, and $R$, where $C_1$ meets $R$ at one node $\sigma'$, $C_2$ meets $R$ at two nodes $q_1$ and $q_2$, $(C_1, \eta_{C_1}) \in \cH^{\hyp}(2g_1 - 2)$ with $\sigma'$ as the zero, $(C_2, \eta_{C_2}) \in \cH^{\hyp}(g_2 - 1, g_2 -1 )$ with $q_1$ and $q_2$ as the two zeros, where $g_1 + g_2 = g-1$, and $R$ contains the unique marking $\sigma$. 
\end{itemize}
\end{lemma} 

See Figure~\ref{fig:hyp-II-1} for the underlying curve $C$ in the three cases above. 
\begin{figure}[h]
    \centering
      \psfrag{1}{$q_1$}
     \psfrag{2}{$q_2$}
      \psfrag{3}{$q_1\sim q_2$}
     \psfrag{t}{$\sigma'$}
      \psfrag{s}{$\sigma$}
       \psfrag{D}{$C_2$}
        \psfrag{C}{$C_1$}
           \psfrag{R}{$R$}
           \psfrag{E}{$E$}
       \includegraphics[scale=0.8]{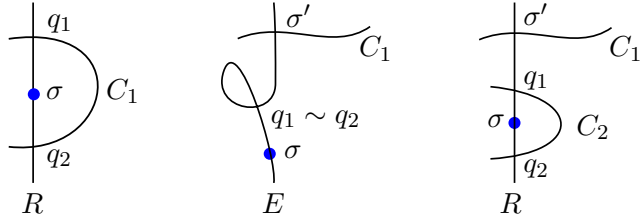}
    \caption{\label{fig:hyp-II-1} The underlying curve $C$ from left to right for cases (1), (2) and (3) in Lemma~\ref{lem:hyp-1}.}
      \end{figure}
 
\begin{proof}
Suppose $(C, \eta)$ is in $\Delta^{\hyp}(2g-2, \cC)$. Then it admits a hyperelliptic admissible double cover $\pi$. Since $C$ has a unique marking $\sigma$, it has only one rational component $R$, and $R$ has to contain $\sigma$. The cover $\pi$ restricted to $R$ has two ramification points, one of which is $\sigma$, and let $\sigma'$ be the other. Denote by $q_1$ and $q_2$ the two polar nodes in $R$ that arise in the description of 
degeneration types (i), (ii), or (iii). By definition of admissible cover, $q_1$ and $q_2$ are hyperelliptic conjugates under $\pi$. Moreover, any tail of $C$ attached to $R$ has to be attached at the ramification point $\sigma'$. 

Based on the above constraints, there are three possibilities for $\pi$ as follows. First, $q_1$ and $q_2$ join $R$ to a different component, and there is no tail attached at $\sigma'$, which gives case (1). On the other hand if there is a tail attached at $\sigma'$, it gives case (3). Finally one can identify $q_1$ and $q_2$ to form a self node of $R$, and attach a tail at $\sigma'$ to ensure that the genus of the total curve is at least two, which gives case (2). By analyzing the corresponding admissible cover in each case, we see that the newly added components along with their differentials satisfy the desired claim. 

Conversely if $(C, \eta)$ is one of the three cases, one can easily construct the corresponding hyperelliptic admissible cover, and we omit the details. 
\end{proof}

\begin{lemma}
\label{lem:hyp-2}
Suppose $(C, \eta)$ is in the principal boundary $\Delta(g-1, g-1, \cC)$ for a configuration $\cC$ of type II. Then $(C, \eta)$ is in 
$\Delta^{\hyp}(g-1, g-1, \cC)$ if and only if it is one of the following types: 
\begin{itemize}
\item[{\rm(1)}] $C$ has three components $C_1$, $R_1$, and $R_2$, where each $R_i$ meets $C_1$ at one node, $R_1$ and $R_2$ meet 
at one node, $(C_1, \eta_{C_1}) \in \cH^{\hyp}(g-2,g-2)$ with the two zeros at the nodes of $C_1$, and each $R_i$ 
contains a marking $\sigma_i$ for $i= 1, 2$. 

\item[{\rm (2)}] $C$ has four components $C_1$, $C_2$, $R_1$, and $R_2$, where each $C_i$ meets each $R_j$ at one node for $i, j = 1, 2$, 
$(C_i, \eta_{C_i}) \in \cH^{\hyp}(g_i-1, g_i-1)$ with the two zeros at the nodes of $C_i$ and $g_1 + g_2 = g-1$, and each $R_j$ contains a marking 
$\sigma_j$. 
\end{itemize}
\end{lemma} 

See Figure~\ref{fig:hyp-II-2} for the underlying curve $C$ in the two cases above. 
\begin{figure}[h]
    \centering
      \psfrag{R}{$R_1$}
     \psfrag{S}{$R_2$}
      \psfrag{C}{$C_1$}
     \psfrag{t}{$\sigma_2$}
      \psfrag{s}{$\sigma_1$}
       \psfrag{D}{$C_2$}           
       \includegraphics[scale=0.8]{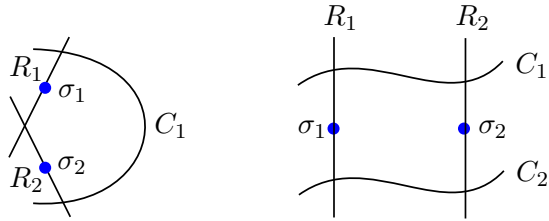}
    \caption{\label{fig:hyp-II-2} The underlying curve $C$ from left to right for cases (1) and (2) in Lemma~\ref{lem:hyp-2}.}
      \end{figure}

\begin{proof}
The proof is similar to the previous one. Suppose $(C, \eta)$ is in $\Delta^{\hyp}(g-1, g-1, \cC)$. Then it admits a hyperelliptic admissible double cover 
$\pi$. Since $\eta$ has two zeros $\sigma_1$ and $\sigma_2$, there are two rational components $R_1$ and $R_2$ in $C$, each containing one zero. Moreover, $\sigma_1$ and $\sigma_2$ are conjugates under $\pi$, hence the degree of $\pi$ restricted to each $R_i$ is one. Consequently there is no tail attached to $R_i$, for otherwise the attaching point in $R_i$ would be a ramification node of $\pi$ by definition of admissible cover. 

Based on the above constraints, there are two possibilities for $\pi$ as follows. Let $p_i$ and $q_i$ be the two nodes of $R_i$. First, if $R_1$ and $R_2$ meet at one node, say, by identifying $p_1$ with $p_2$, then there is another component $C_1$ that joins $R_1$ and $R_2$ at $q_1$ and $q_2$, respectively, which gives case (1). If $R_1$ and $R_2$ are disjoint, then there must be two components $C_1$ and $C_2$, where each $C_i$ connects $R_1$ and $R_2$ at $p_i$ and $q_i$, respectively, which is case (2). Finally notice that $R_1$ and $R_2$ cannot intersect at both nodes, for otherwise there is no other component, and the genus of $C$ would be one. Hence the above two cases are the only possibilities. By analyzing the corresponding admissible cover in each case, we see that the newly added components along with their differentials satisfy the desired claim. 

Conversely if $(C, \eta)$ is one of the two cases, one can easily construct the corresponding hyperelliptic admissible cover, and we omit the details. 
\end{proof}

%%%%%%%%%%%%%%%%%%%%%%%%%%%%
%%%%%%%%%%%%%%%%%%%%%%%%%%%%
\section{Principal boundary for quadratic differentials}
\label{sec:quad}
%%%%%%%%%%%%%%%%%%%%%%%%%%%%
%%%%%%%%%%%%%%%%%%%%%%%%%%%%

In \cite{MZ} Masur and Zorich carried out an analogous description for the principal boundary of moduli spaces of quadratic differentials, which parameterizes quadratic differentials 
with a prescribed generic configuration of short \^{h}omologous saddle connections, where ``\^{h}omologous'' is defined by passing to the canonical double cover (see \cite[Definition 1]{MZ}). The combinatorial structure of configurations of \^{h}omologous saddle connections is described in terms of \emph{ribbon graphs} (see \cite[Figure 6]{MZ}), which can be used as building blocks to construct a flat surface in the principal boundary. 

As the lengths of these \^{h}omologous saddle connections approach zero, we can also describe the principal boundary of limit differentials by using \emph{twisted quadratic differentials} (in the sense of twisted $k$-differentials in \cite{BCGGM2} for $k=2$). The definition of twisted quadratic differentials is almost the same as that of twisted abelian differentials, with one exception that the zero or pole orders on the two branches at every node sum to $-4$.  

Since the idea of describing the principal boundary is similar and only the combinatorial structure gets more involved, we will explain our method by going through a number of examples, in which almost all typical ribbon graphs appear. Consequently the method can be adapted to any given configuration without further difficulties. 

%%%%%%%%%%%%%%%%%%%%%%%%%%
\subsection{Ribbon graphs of configurations}
%%%%%%%%%%%%%%%%%%%%%%%%%%

We briefly recall the geometric meaning of the ribbon graphs (see \cite[Section 1]{MZ} and \cite[Section 2]{Goujard}) for more details). A ribbon graph captures the information of boundary surfaces after removing the \^{h}omologous saddle connections in a given configuration and how these boundary surfaces are glued to form the original surface. A vertex labeled by $\circ$, $\oplus$ or $\ominus$ in the graph represents a cylinder, a boundary surface of trivial holonomy or a boundary surface of non-trivial holonomy, respectively. Here whether or not the holonomy is trivial corresponds to whether or not the quadratic differential is the square of an abelian differential. An edge joining two vertices represents a common saddle connection on the boundaries of the corresponding two surfaces. The boundary of a ribbon graph is decorated by integers that encode the information of cone angles between consecutive \^{h}omologous saddle connections. Each vertex is decorated by a set of integers (possibly empty) that encodes the type of singularities in the interior of the corresponding boundary surface. Connected components of the boundary of a ribbon graph correspond to newborn zeros after gluing the boundary surfaces together. 

%%%%%%%%%%%%%%%%%%%%%%%%%%
\subsection{Configurations in genus $2$}
%%%%%%%%%%%%%%%%%%%%%%%%%%

In \cite[Appendix B]{MZ} Masur and Zorich described explicitly configurations of  \^{h}omologous saddle connections for holomorphic quadratic differentials in genus $2$. Below we will describe the corresponding principal boundary of limit twisted quadratic differentials for the three configurations of the stratum $\cQ(2,2)$ (see \cite[Figure 22]{MZ}). 

The first ribbon graph on the left of \cite[Figure 22]{MZ} corresponds to a flat surface on the left of Figure~\ref{fig:quad-1}. 
 If the saddle connection $\gamma$ shrinks to a point, we obtain a flat surface $(E, \eta_E) \in \cQ(2, -1, -1)$ where the two simple poles are identified as one point. 
 Alternatively, cutting the surface open along $\gamma$, we obtain a surface with two boundary components $\gamma'$ and $\gamma''$. If we expand the neighborhoods of $\gamma'$ and $\gamma''$ to arbitrarily large, it gives a meromorphic quadratic differential $(R, \eta_R) \in \cQ(2, -3, -3)$, since the flat geometric neighborhood of a triple pole of a quadratic differential corresponds to a (broken) half-plane. Combining them together, we conclude that the underlying pointed stable curve $C$ of the limit differential consists of $E$ union $R$ at two nodes, where both $E$ and $R$ contain a marked double zero, see the right side of Figure~\ref{fig:quad-1}.  
 Conversely given such $C$ and $\eta = (\eta_E, \eta_R)$, since $\eta$ is a twisted quadratic differential and satisfies the global residue condition in \cite{BCGGM3}, $(C, \eta)$ can be smoothed into the Masur-Zorich principal boundary for this configuration. 
 
 \begin{figure}[h]
    \centering
     \psfrag{r}{$\gamma$}
      \psfrag{R}{$R$}
       \psfrag{E}{$E$}
    \includegraphics[scale=0.8]{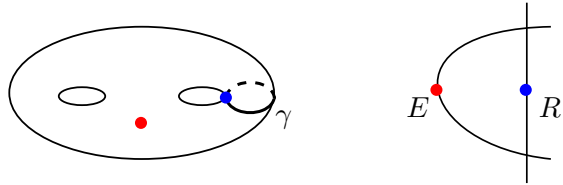}
    \caption{\label{fig:quad-1} The surface corresponding to the first ribbon graph on the left of  \cite[Figure 22]{MZ} and the underlying curve of its degeneration as $\gamma \to 0$.}
      \end{figure}

The second ribbon graph on the left of \cite[Figure 22]{MZ} corresponds to a flat surface on the left of Figure~\ref{fig:quad-2}. 
 When the saddle connections $\gamma_i$ shrink, the three cylinders all become arbitrarily long, hence they give rise to three nodes, each of which is of pole type $(-2, -2)$ in terms of twisted quadratic differentials (or of pole type $(-1, -1)$ in terms of twisted abelian differentials locally). Moreover, the node $q_0$ in the middle is separating, because removing the core curve of the middle cylinder disconnects the surface. Similarly we see that the other two nodes $q_1$ and $q_2$ are non-separating. Therefore, we conclude that the underlying pointed stable curve $C$ of the limit differential consists of two nodal Riemann spheres $R_1$ and $R_2$, where each $(R_i, \eta_i) \in \cQ(2, -2, -2, -2)$ has the last two poles identified as $q_i$ and $R_1, R_2$ are glued by identifying their first poles as $q_0$, see the right side of Figure~\ref{fig:quad-2}.  
 In addition, the half-infinite cylinders corresponding to $\eta_i$ at $q_i$ for $i = 1, 2$ have identical widths, both equal to one-half of the width of the half-infinite cylinders at $q_0$. Conversely given such $C$ and $\eta = (\eta_1, \eta_2)$, the width condition in this case is precisely the matching residue condition in the definition of twisted $k$-differentials in \cite{BCGGM3}, hence $(C, \eta)$ can be smoothed into the Masur-Zorich principal boundary for this configuration. 
 
  \begin{figure}[h]
    \centering
     \psfrag{r}{$\gamma$}
      \psfrag{R}{$R_1$}
       \psfrag{S}{$R_2$}
       \psfrag{a}{$q_1$}
       \psfrag{b}{$q_2$}
       \psfrag{c}{$q_0$}
       \psfrag{1}{$\gamma_1$}
        \psfrag{2}{$\gamma_2$}
        \psfrag{3}{$\gamma_3$}
         \psfrag{4}{$\gamma_4$}
    \includegraphics[scale=0.8]{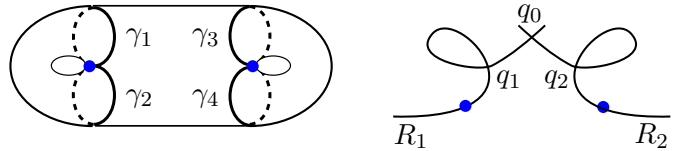}
    \caption{\label{fig:quad-2} The surface corresponding to the second ribbon graph on the left of \cite[Figure 22]{MZ} and the underlying curve of its degeneration as $\gamma_i\to 0$.}
      \end{figure}

The last ribbon graph on the left of \cite[Figure 22]{MZ} corresponds to a flat surface on the left of Figure~\ref{fig:quad-3}. The local picture around each saddle connection $\gamma_i$ is the same as in the first ribbon graph. Hence 
 the underlying pointed stable curve $C$ of the limit differential consists of three rational components $R_0$, $R_1$ and $R_2$, where 
 each of $R_1$ and $R_2$ contains a marked zero and meets $R_0$ at two nodes, see the right side of Figure~\ref{fig:quad-3}.  
 Moreover, the limit twisted quadratic differential $\eta = (\eta_0, \eta_1, \eta_2)$ satisfies that $(R_0, \eta_0) \in \cQ(-1, -1, -1, -1)$ with simple poles at the nodes and $(R_1, \eta_1) \cong (R_2, \eta_2) \in \cQ(2, -3, -3)$ with triple poles at the nodes.  Conversely given such $(C, \eta)$, again by \cite{BCGGM3} it can be smoothed into the Masur-Zorich principal boundary for this configuration. 
 
  \begin{figure}[h]
    \centering
     \psfrag{r}{$\gamma_1$}
       \psfrag{s}{$\gamma_2$} 
       \psfrag{1}{$R_1$}
        \psfrag{2}{$R_2$}
        \psfrag{0}{$R_0$}
           \includegraphics[scale=0.8]{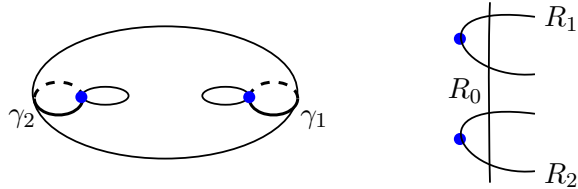}
    \caption{\label{fig:quad-3} The surface corresponding to the last ribbon graph on the left of \cite[Figure 22]{MZ} and the underlying curve of its degeneration as $\gamma_i\to 0$.}
      \end{figure}

 %%%%%%%%%%%%%%%%%%%%%%%%%%
\subsection{A configuration in genus $13$}
%%%%%%%%%%%%%%%%%%%%%%%%%%
 
 We will convince the reader that our method works equally well in the case of high genera by considering an example in genus $13$ in \cite[Figure 7]{MZ}. The underlying pointed stable curve $C$ of the limit differential as $\gamma_i \to 0$ consists of seven components $S_1, \ldots, S_5, R_1, R_2$ meeting as described in Figure~\ref{fig:MZ-13}. 
 
   \begin{figure}[h]
    \centering
     \psfrag{1}{$S_1$}
       \psfrag{2}{$S_2$}
        \psfrag{3}{$S_3$}
    \psfrag{4}{$S_4$}
     \psfrag{5}{$S_5$}
       \psfrag{A}{$R_1$}
        \psfrag{B}{$R_2$}
           \includegraphics[scale=0.8]{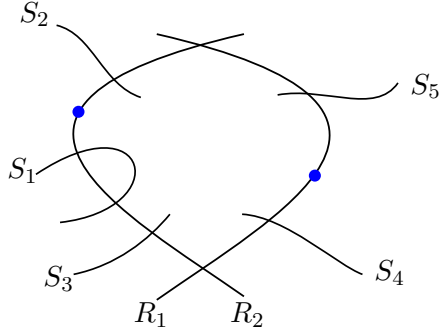}
    \caption{\label{fig:MZ-13} The underlying curve of the degeneration of the surface corresponding to \cite[Figure 7]{MZ} as $\gamma_i\to 0$.}
      \end{figure}

 The $S_i$ components are non-degenerate and carry holomorphic differentials that were already described in \cite[p. 939]{MZ}. The rational component $R_1$ contains the marked zero of order $30$ and 
 carries a differential $\eta_1 \in \cQ(30, -2, -2, -4, -4, -6, -16)$. The other rational component $R_2$ contains the marked zero of order $8$ and 
 carries a differential $\eta_2 \in \cQ(8, -2, -2, -4, -4)$. As before, the half-infinite cylinders for $\eta_1$ and $\eta_2$ at the nodes of their intersection have equal widths. 
 
Let us explain how the components $R_1, R_2$ and their poles appear. The two $\circ$ vertices in the ribbon graph correspond to two cylinders. As they tend to arbitrarily long, we obtain the two nodes with double poles between $R_1$ and $R_2$. Removing $\gamma_4$ and $\gamma_8$ simultaneously  disconnects the curve, which gives rise to $R_1$. Similarly removing $\gamma_5$ and $\gamma_7$ disconnects the curve, hence it gives rise to $R_2$. 
The surface $S_1$ corresponds to the central $\oplus$ vertex, whose boundary has two connected components given by the two connected components 
of its local ribbon graph. Going around each connect boundary component takes a total angle of $2\pi$ by the number decorations, hence the expansion of its local neighborhood to arbitrarily large consists of a pair of (broken) half-planes. It follows that $S_1$ meets $R_1$ at two nodes, both having pole order $4$ for the limit twisted quadratic differential on $R_1$. The intersections of the other $S_j$ with $R_i$ can be analyzed in the same way.

%%%%%%%%%%%%%%%%%
\bibliographystyle{amsalpha}             
\bibliography{Principal}

\newcommand{\etalchar}[1]{$^{#1}$}
\providecommand{\bysame}{\leavevmode\hbox to3em{\hrulefill}\thinspace}
\providecommand{\MR}{\relax\ifhmode\unskip\space\fi MR }
\providecommand{\MRhref}[2]{%
  \href{http://www.ams.org/mathscinet-getitem?mr=#1}{#2}
}
\providecommand{\href}[2]{#2}
\begin{thebibliography}{BCG{\etalchar{+}}16b}

\bibitem[Ati71]{Atiyah}
Michael Atiyah, \emph{Riemann surfaces and spin structures}, Ann. Sci.
  \'{E}cole Norm. Sup. (4) \textbf{4} (1971), 47--62.

\bibitem[BCG{\etalchar{+}}]{BCGGM2}
Matt Bainbridge, Dawei Chen, Quentin Gendron, Samuel Grushevsky, and Martin
  M\"oller, \emph{Boundary structure of strata of abelian differentials}, In
  preparation.

\bibitem[BCG{\etalchar{+}}16a]{BCGGM1}
\bysame, \emph{Compactification of strata of abelian differentials}, 2016,
  arXiv:1604.08834.

\bibitem[BCG{\etalchar{+}}16b]{BCGGM3}
\bysame, \emph{Strata of $k$-differentials}, 2016, arXiv:1610.09238.

\bibitem[BG15]{BG}
Max Bauer and Elise Goujard, \emph{Geometry of periodic regions on flat
  surfaces and associated siegel-veech constants}, Geom. Dedicata \textbf{174}
  (2015), 203--233.

\bibitem[Boi15]{Boissy}
Corentin Boissy, \emph{Connected components of the strata of the moduli space
  of meromorphic differentials}, Comment. Math. Helv. \textbf{90} (2015),
  no.~2, 255--286.

\bibitem[CC16]{CC}
Dawei Chen and Qile Chen, \emph{Spin and hyperelliptic structures of log
  twisted abelian differentials}, 2016, arXiv:1610.05345.

\bibitem[Che15]{ChenDiff}
Dawei Chen, \emph{Degeneration of abelian differentials}, 2015,
  arXiv:1504.01983.

\bibitem[Che16]{Bootcamp}
\bysame, \emph{{Teichm\"uller dynamics in the eyes of an algebraic geometer}},
  2016, arXiv:1602.02260.

\bibitem[Cor89]{Cornalba}
Maurizio Cornalba, \emph{Moduli of curves and theta-characteristics}, Lectures
  on Riemann surfaces, Proceedings of the First College on Riemann Surfaces
  held in Trieste, November 9--December 18, 1987, World Sci. Publ., Teaneck,
  NJ, 1989, pp.~560--589.

\bibitem[EM01]{EM}
Alex Eskin and Howard Masur, \emph{{Asymptotic formulas on flat surfaces}},
  Ergodic Theory Dynam. Systems \textbf{21} (2001), no.~2, 443--478.

\bibitem[EMZ03]{EMZ}
Alex Eskin, Howard Masur, and Anton Zorich, \emph{Moduli spaces of {A}belian
  differentials: the principal boundary, counting problems, and the
  {S}iegel-{V}eech constants}, Publ. Math. Inst. Hautes \'{E}tudes Sci.
  \textbf{97} (2003), 61--179.

\bibitem[FP15]{FP}
Gavril Farkas and Rahul Pandharipande, \emph{The moduli space of twisted
  canonical divisors, with an appendix by {F}elix {J}anda, {R}ahul
  {P}andharipande, {A}aron {P}ixton, and {D}imitri {Z}vonkine}, 2015,
  arXiv:1508.07940.

\bibitem[Gen15]{Gendron}
Quentin Gendron, \emph{The {D}eligne-{M}umford and the incidence variety
  compactifications of the strata of {$\Omega\mathcal{M}_{g}$}}, 2015,
  arXiv:1503.03338.

\bibitem[Gou]{Goujard}
Elise Goujard, \emph{Siegel-veech constants for strata of moduli spaces of
  quadratic differentials}, Geom. Funct. Anal. \textbf{25}, no.~5, 1440--1492.

\bibitem[Gu{\'e}16]{Guere}
J{\'e}r{\'e}my Gu{\'e}r{\'e}, \emph{{A generalization of the double
  ramification cycle via log-geometry}}, 2016, arXiv:1603.09213.

\bibitem[HM82]{HM}
Joe Harris and David Mumford, \emph{On the {K}odaira dimension of the moduli
  space of curves}, Invent. Math. \textbf{67} (1982), no.~1, 23--88, With an
  appendix by William Fulton.

\bibitem[KZ03]{KZ}
Maxim Kontsevich and Anton Zorich, \emph{Connected components of the moduli
  spaces of {A}belian differentials with prescribed singularities}, Invent.
  Math. \textbf{153} (2003), no.~3, 631--678.

\bibitem[Mum71]{Mumford}
David Mumford, \emph{Theta characteristics of an algebraic curve}, Ann. Sci.
  \'{E}cole Norm. Sup. (4) \textbf{4} (1971), 181--192.

\bibitem[MZ08]{MZ}
Howard Masur and Anton Zorich, \emph{{Multiple saddle connections on flat
  surfaces and the principal boundary of the moduli spaces of quadratic
  differentials}}, Geom. Funct. Anal. \textbf{18} (2008), no.~3, 919--987.

\bibitem[Vee98]{Veech}
William~A. Veech, \emph{{Siegel measures}}, Ann. of Math. (2) \textbf{148}
  (1998), no.~3, 895--944.

\bibitem[WM15]{MW}
Alex Wright and Maryam Mirzakhani, \emph{{The boundary of an affine invariant
  submanifold}}, 2015, arXiv:1508.01446.

\bibitem[Wri15]{Wright}
Alex Wright, \emph{{Translation surfaces and their orbit closures: an
  introduction for a broad audience}}, EMS Surv. Math. Sci. \textbf{2} (2015),
  no.~1, 63--108.

\bibitem[Zor06]{Zorich}
Anton Zorich, \emph{Flat surfaces.}, {Frontiers in number theory, physics, and
  geometry I. On random matrices, zeta functions, and dynamical systems. Papers
  from the meeting, Les Houches, France, March 9--21, 2003}, Berlin: Springer,
  2nd printing ed., 2006, pp.~437--583.

\end{thebibliography}

\end{document}